\documentclass[11pt]{article}

\usepackage[english]{babel}
\usepackage{amssymb}
\usepackage{amsmath}
\usepackage{enumerate}
\usepackage{amsthm}
\usepackage{amsfonts}
\usepackage{listings}
\usepackage{xypic}
\usepackage{hyperref}
\usepackage{rotating}

\usepackage{MnSymbol}

\title{Higher stacks as diagrams}
\date{April 6, 2022}
\author{Fritz H\"ormann\\ Mathematisches Institut, Albert-Ludwigs-Universit\"at Freiburg}

\usepackage[numbers,sort&compress]{natbib}

\usepackage{color}
\usepackage{graphicx}

\usepackage{tikz}

\newcommand{\comment}[1]{}

\newtheorem{SATZ}{Theorem}[section]

\newtheorem{LEMMA}[SATZ]{Lemma}
\newtheorem{DEF}[SATZ]{Definition}
\newtheorem{PROP}[SATZ]{Proposition}

\newtheorem{KOR}[SATZ]{Corollary}

\newtheorem{BEM}[SATZ]{Remark}

\newtheoremstyle{bare}        
  {}            
  {}            
  {\normalfont}                 
  {}                            
  {\bfseries}                   
  {}                            
  {.0em}                           
  {\thmnumber{#2}#1. \thmnote{\normalfont\textsc{(#3)}} } 

\theoremstyle{bare}
\newtheorem{PAR}[SATZ]{}

\newcommand{\N}{ \mathbb{N} }
\newcommand{\HH}{ \mathbb{H} }
\newcommand{\DD}{ \mathbb{D} }
\newcommand{\EE}{ \mathbb{E} }

\newcommand{\SSS}{ \mathbb{S} }

\DeclareMathOperator{\proj}{proj}

\DeclareMathOperator{\spl}{split}

\DeclareMathOperator{\Fun}{Fun}
\DeclareMathOperator{\Fib}{Fib}
\DeclareMathOperator{\Cof}{Cof}

\DeclareMathOperator{\Mor}{Mor}

\DeclareMathOperator{\id}{id}

\DeclareMathOperator{\Hom}{Hom}

\DeclareMathOperator{\Dia}{Cat}
\DeclareMathOperator{\cor}{cor}

\DeclareMathOperator{\op}{op}
\DeclareMathOperator{\cart}{cart}
\DeclareMathOperator{\cocart}{cocart}
\DeclareMathOperator{\pr}{pr}
\DeclareMathOperator{\colim}{colim}

\DeclareMathOperator{\hocolim}{hocolim}
\DeclareMathOperator{\holim}{holim}

\newcommand{\tw}[1]{ {{}^{\downarrow \uparrow} #1 }}

\newcommand{\twc}[1]{ {{}^{\downarrow \uparrow \downarrow} #1 }}

\begin{document}

\maketitle



\section*{Abstract}

Several possible presentations for the homotopy theory of (non-hypercomplete) $\infty$-stacks on a classical site $\mathcal{S}$ are discussed. 
In particular, it is shown that an elegant combinatorial description in terms of diagrams in $\mathcal{S}$ exists, similar to Cisinski's presentation, based on work of Quillen, Thomason and Grothendieck, of usual homotopy theory by small categories and their smallest (basic) localizer. As an application it is shown that any (local) fibered (a.k.a.\@ algebraic) derivator over $\mathcal{S}$ with stable fibers extends to $\infty$-stacks in a well-defined way under mild assumptions.  

\tableofcontents

\section{Introduction}

\begin{PAR}
Let $\mathcal{S}$ be a (locally small) category with finite limits and Grothendieck topology. We discuss several models for the homotopy
theory of higher stacks on $\mathcal{S}$. If $\mathcal{S}$ is small, the presentation by the \v{C}ech localization of simplicial presheaves is well-known. 
In case of the trivial topology or the extensive topology it can also be presented by simplicial objects in $\mathcal{S}^\amalg$ (free coproduct completion), or in $\mathcal{S}$, respectively, regardless of the size of $\mathcal{S}$, using  model category structures on simplicial objects in $\mathcal{S}$ constructed in a previous article \cite{Hor21}.
The novelty is a presentation by diagrams in $\mathcal{S}$ together with their smallest {\em localizer}. The notion of basic localizer\footnote{``localisateur fondamental'' in French} was introduced by Grothendieck for small categories
and extended by the author \cite{Hor15} to diagrams in a site $\mathcal{S}$. 
This theory gives an elegant way of extending any fibered derivator with stable fibers on $\mathcal{S}$ (or $\mathcal{S}^{\op}$) to higher stacks using the theory of (co)homological descent for fibered derivators \cite{Hor15}. This is analogous to the universal property, proven by Cisinski \cite{Cis08}, enjoyed by the associated derivator. That property will also be reproven here for the case that $\mathcal{S}$ is arbitrary (not necessarily small or even locally small). 
We emphasize that the theory of localizers on diagrams in $\mathcal{S}$ is not based on the simplex category (or any other ``test category'') at all. 

In a subsequent article \cite{Hor21d} it will be shown that also derivator six-functor-formalisms (which are certain fibered multiderivators over the 2-multicategory $\mathcal{S}^{\cor}$ of correspondences in $\mathcal{S}$) extend to higher stacks under
similar hypotheses. There, in contrast to the extension discussed here, some ``algebraicity condition'' on the stack is needed in addition. 
\end{PAR}

\begin{PAR}Recall \cite{Cis03, Hor15} that a (basic) localizer on the category $\Dia$ of small categories is a subclass $\mathcal{W} \subset \Mor(\Dia)$ subject to the following axioms:
\begin{enumerate}
\item[(L1)] $\mathcal{W}$ is weakly saturated (cf.\@ Definition~\ref{DEFWS}); 
\item[(L2 left)] For every small category $I$ with final object the projection functor $I \rightarrow \cdot$ is in $\mathcal{W}$; 
\item[(L3 left)] If a commutative diagram in $\Dia$
\[ \xymatrix{ I \ar[rr]^w \ar[rd] & & J \ar[ld] \\ & K } \]
is such that for all $k \in K$ the induced functor \[ w_k: I \times_{/K} k \rightarrow J \times_{/K} k \] is in $\mathcal{W}$, then the functor $w$ is in $\mathcal{W}$.
\end{enumerate}
There is an obvious dual notion of colocalizer with axioms (L2 right) and (L3 right). However, it can be shown \cite{Cis04} that any localizer is a colocalizer and vice versa. 
The class of localizers is clearly stable under intersection and thus there is a smallest localizer $\mathcal{W}_{\infty}$.
\end{PAR}
\begin{SATZ}[Cisinski \cite{Cis04}]
\[ \mathcal{W}_{\infty} = \{ w \in \Mor(\Dia) \ | \ N(w) \text{ is a weak equivalence of simplicial sets }\} \]
\end{SATZ}
The proof shows in addition that there is an equivalence of categories with weak equivalences induced by $N$:
\begin{equation}\label{equiv1} (\Dia, \mathcal{W}_{\infty}) \cong (\mathcal{SET}^{\Delta^{\op}}, \mathcal{W}) \end{equation}
where $\mathcal{W}$ is the class of usual weak equivalences of simplicial sets. Hence the two derivators, or $\infty$-categories, formed by these two models
are equivalent.

\begin{PAR} The notions of localizer and colocalizer have been extended in \cite{Hor15} to the category $\Dia(\mathcal{S})$ (resp.\@ $\Dia^{\op}(\mathcal{S})$) of diagrams in $\mathcal{S}$ (cf.\@ Definition~\ref{DEFDIA})  as follows: Fix a Grothendieck pretopology on $\mathcal{S}$. 
A {\bf localizer} on $\Dia(\mathcal{S})$ is a subclass $\mathcal{W} \subset \Mor(\Dia(\mathcal{S}))$ subject to the following axioms: 
\begin{enumerate}
\item[(L1)] $\mathcal{W}$ is weakly saturated (cf.\@ Definition~\ref{DEFWS}); 
\item[(L2 left)] For every diagram $(I, S)$ such that $I$ has a final object $i$ the morphism $(I, S) \rightarrow (\cdot, S(i))$ is in $\mathcal{W}$; 
\item[(L3 left)] If a commutative diagram in $\Dia(\mathcal{S})$
\[ \xymatrix{ (I, S) \ar[rr]^w \ar[rd] & & (J, T) \ar[ld] \\ & (K, U) } \]
is such that for any $k \in K$ there is a covering $\{U_{k, i} \rightarrow U(k) \}$ such that
the induced morphism \[ w_{k, i}: I \times_{/(K, U)} (k, U_{k,i}) \rightarrow J \times_{/(K, U)} (k, U_{k,i}) \] is in $\mathcal{W}$ for all $i$, then the morphism $w$ is in $\mathcal{W}$.
\item[(L4 left)] If $w: (I, \alpha^*T) \rightarrow (J, T)$ is a morphism (of pure diagram type) in $\Dia(\mathcal{S})$ and
for all $j$ the morphism
\[ (j \times_{/J} I, \cdot) \rightarrow (\cdot, \cdot)    \]
is in $\mathcal{W}$ then the morphism $w$ is in $\mathcal{W}$. 
\end{enumerate}
There are again obvious dual axioms (L1--L4 right) defining a {\bf colocalizer} on $\Dia^{\op}(\mathcal{S}) \cong \Dia(\mathcal{S}^{\op})^{2-\op}$, giving a directly incomparable notion in this case\footnote{however, a localizer on $\Dia(\mathcal{S})$ gives rise to a colocalizer on  $\Dia^{\op}(\mathcal{S}^{\op})$ under the identification $(I, S) \mapsto (I^{\op}, S^{\op})$ and vice versa.}. 
If $\mathcal{S}=\cdot$ is the terminal category then (L4 left) follows from (L1--L3 left) and similarly (L4 right) follows from (L1--L3 right) and furthermore, as mentioned above, 
the sets of axioms (L1--L3 left) and (L1--L3 right) are equivalent using the natural identification ({\em not} involving taking the opposite) $\Dia(\cdot) = \Dia = \Dia^{\op}(\cdot)$. Thus both recover the notion of basic localizer of Grothendieck. 

Again, there exists a smallest localizer $\mathcal{W}_\infty$. 
In this article we prove
\end{PAR} 

\begin{SATZ}[cf.\@ Theorem~\ref{THEOREMINTWECOFCECH}]
If $\mathcal{S}$ is small, we have
\[ \mathcal{W}_{\infty} = \{ w \in \Mor(\Dia(\mathcal{S})) \ | \ N(w) \text{ is a \v{C}ech weak equivalence of simplicial presheaves }\} \]
\end{SATZ}
For the extension of the functor {\em nerve} $N$ to $\Dia(\mathcal{S})$ see Definition~\ref{DEFN}. 

Caution: Axiom (L4 left\nobreak/right) has been stated in \cite{Hor15} in a slightly different form --- the difference does not matter for the {\em smallest} localizer because all that is needed
for the proof is the following: for morphisms $I \rightarrow J$ such that, for all $j \in J$, the diagram $j \times_{/J} I$ is contractible in the sense of the {\em smallest} localizer on $\Dia$, morphisms of the form
$(I, \alpha^* T) \rightarrow (J, T)$ are in $\mathcal{W}$, cf.\@ the discussion in the proof of Theorem~\ref{SATZUNIV}.

Still assuming that $\mathcal{S}$ is small, there is also, generalizing (\ref{equiv1}), an equivalence of categories with weak equivalences: 
\[ (\Dia(\mathcal{S}), \mathcal{W}_{\infty}) \cong (\mathcal{SET}^{\mathcal{S}^{\op} \times \Delta^{\op}}, \mathcal{W}_{loc}) \]
where  $\mathcal{W}_{loc}$ is the class of \v{C}ech weak equivalences, i.e.\@ the class of weak equivalences in the left Bousfield localization of simplicial presheaves at the \v{C}ech covers. 
There exists a further left Bousfield localization of $\mathcal{SET}^{\mathcal{S}^{\op} \times \Delta^{\op}}$ at all hypercovers (modeling hypercomplete $\infty$-stacks) and it would be interesting to investigate how the corresponding localizer on $\Dia(\mathcal{S})$ can be axiomatically described.  

\begin{PAR} For a trivial Grothendieck topology we have (potentially) four categories with weak equivalences
\[ (\mathcal{S}^{\Delta^{\op}}, \mathcal{W})  \quad  (\mathcal{S}^{\amalg, \Delta^{\op}}, \mathcal{W})  \quad (\mathcal{SET}^{\mathcal{S}^{\op} \times \Delta^{\op}}, \mathcal{W})  \quad (\Dia(\mathcal{S}), \mathcal{W}_\infty)  \]
where $\mathcal{S}^{\amalg}$ is the free coproduct completion.  The first three are part of model category structures (not having necessarily all limits and colimits) and will be recalled in section~\ref{SECTMC}. 
For the existence of the structure on $\mathcal{SET}^{\mathcal{S}^{\op} \times \Delta^{\op}}$ (simplicial presheaves) one has to assume that $\mathcal{S}$ is small, for the structure on $\mathcal{S}^{\amalg, \Delta^{\op}}$ that $\mathcal{S}$ has finite limits, 
and for the structure on $\mathcal{S}^{\Delta^{\op}}$ that $\mathcal{S}$ has finite limits, is extensive (and thus is big in general) and every object is a coproduct of $\N$-small objects. 
There is a Quillen adjunction between $\mathcal{S}^{\Delta^{\op}}$ and $\mathcal{S}^{\amalg, \Delta^{\op}}$ (when both exist)  which induces an equivalence between the former  
and the left Bousfield localization of the latter at the (\v{C}ech) covers for the extensive topology. If $\mathcal{S}$ is small we have equivalences
\[ (\mathcal{S}^{\amalg, \Delta^{\op}}, \mathcal{W}) \cong  (\mathcal{SET}^{\mathcal{S}^{\op} \times \Delta^{\op}}, \mathcal{W}) \cong (\Dia(\mathcal{S}), \mathcal{W}_{\infty}) \]
where $\mathcal{W}_{\infty}$ is the smallest localizer for the trivial topology, and if $\mathcal{S}$ is not necessarily small, with finite limits, we have at least an equivalence
\[ (\mathcal{S}^{\amalg, \Delta^{\op}}, \mathcal{W}) \cong (\Dia(\mathcal{S}), \mathcal{W}_{\infty}) \]
where $\mathcal{W}_{\infty}$ is again the smallest localizer for the trivial topology, and if $\mathcal{S}$ is extensive and has the above properties, we have equivalences
\[ (\mathcal{S}^{\Delta^{\op}}, \mathcal{W}) \cong  (\mathcal{S}^{\amalg, \Delta^{\op}}, \mathcal{W}_{loc}) \cong (\Dia(\mathcal{S}), \mathcal{W}_{\infty}) \]
where $\mathcal{W}_{loc}$ is the class of weak equivalences in the left Bousfield localization w.r.t.\@ the extensive topology (which exists regardless of the size of $\mathcal{S}$) and $\mathcal{W}_{\infty}$ is
the smallest localizer for the extensive topology. 
\end{PAR} 

\begin{PAR}Let $\mathcal{S}$ be a small category with finite limits and Grothendieck pretopology. 
As an application we show in Section~\ref{SECTEX} that any nice enough fibered derivator over $\mathcal{S}$ which is
local w.r.t.\@ the Grothendieck pretopology on $\mathcal{S}$ (cf.\@ Definition~\ref{DEFLOCAL}) has a natural extension to
$\mathbb{H}^2(\mathcal{SET}^{\Delta^{\op} \times \mathcal{S}^{\op}}_{loc})$, the homotopy pre-2-derivator of the model category 
$\mathcal{SET}^{\Delta^{\op} \times \mathcal{S}^{\op}}_{loc}$. The main ingredient is the theory of strong (co)homological descent
developed in \cite{Hor15} mainly for this purpose.
This is merely an example --- in a subsequent article \cite{Hor21d} we will come back to this question and will also construct extensions of derivator six-functor-formalisms to higher stacks.
We concentrate on the ``homological case''  and leave the dual ``cohomological case'' (starting with a fibered derivator over $\mathcal{S}^{\op}$) to the reader. 
The precise statement is as follows:
\end{PAR}

\vspace{0.2cm}
{\bf Theorem~\ref{SATZEXDER}.\@}
Let $\DD \rightarrow \SSS$ be an infinite fibered derivator which is local w.r.t.\@ the Grothendieck
pretopology on $\mathcal{S}$ with stable, well-generated fibers. Then there is a natural extension 
\[ \DD'' \rightarrow \mathbb{H}^2(\mathcal{SET}^{\Delta^{\op} \times \mathcal{S}^{\op}}_{loc}) \]
such that the pullback of $\DD''$ along the natural morphism 
\[ \SSS \rightarrow \mathbb{H}^2(\mathcal{SET}^{\Delta^{\op} \times \mathcal{S}^{\op}}_{loc})  \]
is equivalent to $\DD$ and such that for a pair
$(I, S) \in \Dia(\mathcal{S})$ with $I \in \Dia$ and $S \in \mathcal{S}^{I}$
we have an equivalence of derivators: 
\[ \DD''_{N(I, S)} \cong  \DD_{(I, S)}^{\cart}. \]
\vspace{0.2cm}

\section{Presentations of  higher stacks}\label{SECTMC}

We adopt the following conventions: 
\begin{DEF}
A {\bf category with weak equivalences} $(\mathcal{M}, \mathcal{W})$ is a category $\mathcal{M}$ with a subclass $\mathcal{W} \subset \mathrm{Mor}(\mathcal{M})$ of {\em weak equivalences} satisfying 2-out-of-3 and containing all isomorphisms. 

A functor $F: \mathcal{M}_1 \rightarrow \mathcal{M}_2$ is called a {\bf functor of categories with weak equivalences} $(\mathcal{M}_1, \mathcal{W}_1) \rightarrow (\mathcal{M}_2, \mathcal{W}_2)$, if
$F(\mathcal{W}_1) \subseteq \mathcal{W}_2$. It is called an {\bf equivalence of categories with weak equivalences}, if there is a functor of categories with weak equivalences  $G: (\mathcal{M}_2, \mathcal{W}_2) \rightarrow (\mathcal{M}_1, \mathcal{W}_1)$ such that \[F \circ G \cong \id \quad \text{ and } \quad G \circ F \cong \id\] in $\Fun(\mathcal{M}_1, \mathcal{M}_1)[\mathcal{W}_{\mathcal{M}_1}^{-1}]$ (resp.\@ in $\Fun(\mathcal{M}_2, \mathcal{M}_2)[\mathcal{W}_{\mathcal{M}_2}^{-1}]$), where $\mathcal{W}_{\mathcal{M}_1}$ (resp.\@ $\mathcal{W}_{\mathcal{M}_2}$)  is the class of natural transformations which are object-wise weak equivalences .
\end{DEF}

Note that, a priori, equivalent categories with weak equivalences define equivalent $\infty$-categories and equivalent prederivators.  

\begin{DEF}
A {\bf model category} $(\mathcal{M}, \Cof, \Fib, \mathcal{W})$ is a category $\mathcal{M}$ with three subclasses of morphisms such that
\begin{enumerate}
\item $(\mathcal{M}, \mathcal{W})$ is a category with weak equivalences.
\item $(\Cof, \Fib \cap \mathcal{W})$ and $(\Cof \cap \mathcal{W}, \Fib)$ are weak factorization systems (cf.\@ \cite[Definition 3.5]{Hor21b}). 
\end{enumerate}
We do not necessarily assume that $\mathcal{M}$ has all limits and colimits but will always assume that finite limits and all coproducts exist. Furthermore we assume that
push-outs along cofibrations and transfinite compositions of cofibrations exit. 

A {\bf simplicial model category} is in addition simplicially enriched such that a tensoring
\[ \otimes: \mathcal{SET}^{\Delta^{\op}} \times \mathcal{M} \rightarrow \mathcal{M} \]
exists and is left Quillen on all simplicial sets and its right adjoint
\[ \mathcal{HOM}: (\mathcal{SET}^{\Delta^{\op}}_{fin})^{\op} \times \mathcal{M} \rightarrow \mathcal{M} \]
exists on the restriction to finite simplicial sets. 
\end{DEF}

Let $\mathcal{S}$ be a category. Denote by $\mathcal{S}^\amalg$ the free coproduct completion and by $\mathcal{SET}^{\mathcal{S}^{\op}}$ the category of presheaves on  $\mathcal{S}$. There are obvious fully-faithful embeddings 
\[ \mathcal{S} \hookrightarrow \mathcal{S}^\amalg \hookrightarrow \mathcal{SET}^{\mathcal{S}^{\op}}. \]
The embedding $\mathcal{S} \hookrightarrow \mathcal{S}^\amalg$ has a left adjoint if $\mathcal{S}$ has all coproducts while the embedding 
$\mathcal{S} \hookrightarrow \mathcal{SET}^{\mathcal{S}^{\op}}$ has a left adjoint if $\mathcal{S}$ has all colimits. The latter case will not be assumed though at any place in this article. 

\begin{SATZ}\label{SATZSOMODEL}If $\mathcal{S}$ has finite limits,
$\mathcal{S}^{\amalg, \Delta^{\op}}$ is a simplicial model category (with finite limits and all coproducts) in which the fibrations and weak equivalences are 
the morphisms such that $\Hom(S, -)$ is a fibration, respectively a weak equivalence of simplicial sets for all $S \in \mathcal{S}$. The cofibrant objects are those
simplicial objects in which the degeneracies induce a decomposition:
\[ X_n \cong X_{n,nd} \amalg \coprod_{\substack{\Delta_n \twoheadrightarrow \Delta_m \\ n \not= m}} X_{m, nd}. \]
We call the model category structure the {\bf split-projective} structure.
\end{SATZ}
\begin{proof}\cite[Theorem~6.1]{Hor21}. Cf.\@ also \cite[Theorem~4.9]{Hor21} for the description of cofibrant objects.
\end{proof}

\begin{BEM}
In the language of \cite{Hor21} the model structure on $\mathcal{S}^{\amalg, \Delta^{\op}}$ is transported in the sense of \cite[4.4]{Hor21}
from the weak factorization system $(\mathcal{L}_{\proj,\spl}, \mathcal{R}_{\proj, \spl})$ on $\mathcal{S}^{\amalg}$ (cf.\@ \cite[3.11]{Hor21}).
\end{BEM}

\begin{SATZ}\label{SATZSPMODEL}
If $\mathcal{S}$ is small,
$\mathcal{SET}^{\mathcal{S}^{\op} \times \Delta^{\op}}$ is a simplicial model category (with all limits and colimits) in which the fibrations and weak equivalences are 
the morphisms of presheaves such that $\Hom(h_S, -)$ is a fibration, respectively a weak equivalence of simplicial sets for all $S \in \mathcal{S}$. 
The cofibrations are those morphisms $X \rightarrow Y$ such that the morphism 
\[ L_n Y \amalg_{L_n X} X_n \rightarrow Y_n  \]
in which $L_n$ denotes the latching object functor of the Reedy structure on $\Delta^{\op}$, is of the form $A \rightarrow \coprod B_i$ in which the $B_i$ are retracts of representables (i.e.\@ themselves representable if $\mathcal{S}$ has finite limits). 
The model category is cofibrantly generated and left and right proper. 
\end{SATZ}
See e.g.\@ \cite[Theorem~4.9]{Hor21} for the description of cofibrations. This is called the {\bf projective} model structure on simplicial presheaves. 

\begin{BEM}
In the language of \cite{Hor21} the model structure on $\mathcal{SET}^{\mathcal{S}^{\op} \times \Delta^{\op}}$ is transported in the sense of \cite[4.4]{Hor21}
from the weak factorization system $(\mathcal{L}_{\proj}, \mathcal{R}_{\proj})$ on $\mathcal{SET}^{\mathcal{S}^{\op}}$ \cite[3.11]{Hor21}.
\end{BEM}

\begin{PAR}\label{PARR}
There is a functor induced by the Yoneda embedding
\[ \xymatrix{ \mathcal{S}^{\amalg, \Delta^{\op}}  \ar[rr]^-R & &   \mathcal{SET}^{\mathcal{S}^{\op} \times \Delta^{\op}}} \]
more precisely, defined by 
\[ R(\{U_i\}_i) (S) := \prod_i \Hom(S, U_i).  \]
\end{PAR}

\begin{PROP}\label{PROPSMALLS}Let $\mathcal{S}$ be a small category with finite limits, and let $R$ be the functor defined in \ref{PARR}. 
\begin{enumerate}
\item 
 $R$ preserves cofibrations, fibrations and weak equivalences. 
\item 
 $R$ induces an equivalence of categories with weak equivalences
\[ \xymatrix{ (\mathcal{S}^{\amalg, \Delta^{\op}}, \mathcal{W}) \ar[r]^-\sim &   (\mathcal{SET}^{\mathcal{S}^{\op} \times \Delta^{\op}}, \mathcal{W}) } \]
\end{enumerate}
\end{PROP}
\begin{proof}\cite[Proposition~6.6]{Hor21}.
\end{proof}

Sometimes it is convenient to work with simplicial diagrams in $\mathcal{S}$ rather than in $\mathcal{S}^\amalg$: 

\begin{SATZ}\label{SATZEXTENSIVEMODELCAT}
Let $\mathcal{S}$ be an extensive category with finite limits such that every object is a coproduct of $\N$-small objects. 
Then $\mathcal{S}^{\Delta^{\op}}$ is a simplicial model category (with finite limits and all coproducts) in which the fibrations and weak equivalences are 
the morphisms such that $\Hom(S, -)$ is a fibration of simplicial sets for all $S \in \mathcal{S}$, respectively a weak equivalence of simplicial sets for all $\N$-small $S \in \mathcal{S}$. 
The cofibrant objects are those simplicial objects in which the degeneracies induce a decomposition
\[ X_n \cong X_{n,nd} \amalg \coprod_{\substack{\Delta_n \twoheadrightarrow \Delta_m \\ n \not= m}} X_{m, nd}. \]
We call the model category structure the {\bf split-projective} structure.
Furthermore, the left Bousfield localization $\mathcal{S}^{\amalg, \Delta^{\op}}_{loc}$  of $\mathcal{S}^{\amalg, \Delta^{\op}}$ at those morphisms that become weak equivalences in $\mathcal{S}^{\Delta^{\op}}$ exists, and we have
Quillen adjunctions
\[ \xymatrix{ \mathcal{S}^{\Delta^{\op}} \ar@<3pt>[rr]^-R & & \ar@<3pt>[ll]^-{\amalg} \mathcal{S}^{\amalg, \Delta^{\op}}_{loc} 
\ar@<3pt>[rr]^-{\id} & & \ar@<3pt>[ll]^-{\id} \mathcal{S}^{\amalg, \Delta^{\op}} } \]
The left hand side adjunction is a Quillen equivalence. 
\end{SATZ}

\begin{proof}
\cite[Proposition~6.5]{Hor21}.
\end{proof}

\begin{BEM}
In the language of \cite{Hor21} the model structure on $\mathcal{S}^{\Delta^{\op}}$ is transported in the sense of \cite[4.4]{Hor21}
from the weak factorization system $(\mathcal{L}_{\proj,\spl}, \mathcal{R}_{\proj,\spl})$ on $\mathcal{S}$ \cite[3.11]{Hor21}.
\end{BEM}

\begin{PAR}
Let $\mathcal{S}$ be a category with Grothendieck pretopology. 
We briefly recall the \v{C}ech localization of a model category of simplicial presheaves on $\mathcal{S}$. For a covering $\{U^{(i)} \rightarrow X\}$ the {\bf \v{C}ech cover} $U_\bullet \rightarrow X$ is a morphism in $\mathcal{S}^{\amalg, \Delta^{\op}}$ defined by 
\[ U_n := \underbrace{U \times_X \cdots \times_X U}_{n-\text{times}} \]
where $U:=\coprod_i U^{(i)}$ in $\mathcal{S}^{\amalg}$. The existence of the fiber product is implied by the axioms of a pretopology. 
Its image in $\mathcal{SET}^{\mathcal{S}^{\op} \times \Delta^{\op}}$ will also be called a \v{C}ech cover. 
\end{PAR}

\begin{PAR}
Consider the category $\mathcal{M}:=\mathcal{SET}^{\mathcal{S}^{\op} \times \Delta^{\op}}$ of simplicial presheaves on $\mathcal{S}$ equipped with the projective model category structure from Theorem~\ref{SATZSPMODEL}. Consider the set $S$ of \v{C}ech covers in $\mathcal{SET}^{\mathcal{S}^{\op} \times \Delta^{\op}}$. An object $Z$ in $\mathcal{M}$  is called {\bf $S$-local}, if for every morphism $f: X \rightarrow Y$ in $\mathcal{M}$ the induced morphism of mapping spaces\footnote{$\mathrm{Map}(X, Y):= [\Hom(X', Y')]$ where $X'$ is a cofibrant replacement of $X$, and $Y'$ is a fibrant replacement of $Y$. }
\[ \mathrm{Map}(Y, Z) \rightarrow \mathrm{Map}(X, Z)  \] 
is a homotopy equivalence. A morphism $f: X \rightarrow Y$ in $\mathcal{M}$ is called an {\bf $S$-equivalence}, if for every $S$-local object $Z$ the morphism
\[ \mathrm{Map}(Y, Z) \rightarrow \mathrm{Map}(X, Z)  \]
is a homotopy equivalence.

A model category structure  $\mathcal{M}_{loc}$ on the same underlying category, in which the cofibrations are the same, and the weak equivalences are the $S$-equivalences is called a left Bousfield localization at $S$. 
\end{PAR}

\begin{SATZ}\label{SATZBLEXACT}
Let $\mathcal{S}$ be a small category with Grothendieck pretopology. 
Then the left Bousfield localization $(\mathcal{SET}^{\mathcal{S}^{\op} \times \Delta^{\op}}, \Cof, \Fib_{loc}, \mathcal{W}_{loc})$ of 
$(\mathcal{SET}^{\mathcal{S}^{\op} \times \Delta^{\op}}, \Cof, \Fib, \mathcal{W})$ at the \v{C}ech covers exists and 
is a left and right proper simplicial model category (with all limits and colimits).

We have a Quillen adjunction
\[ \xymatrix{ \mathcal{SET}_{loc}^{\mathcal{S}^{\op} \times \Delta^{\op}} \ar@<5pt>[rr]^{R} & & \ar@<5pt>[ll]^{L} \mathcal{SET}^{\mathcal{S}^{\op} \times \Delta^{\op}}   } \]
such that the underlying functors for $L$ and $R$ are the identity. In this case the localization is {\bf exact} in the sense that the left adjoint $L$ also commutes with homotopically finite homotopy limits.
\end{SATZ}
\begin{proof} 
This is well-known. See e.g.\@ \cite{Hir03} for the existence of the localization and \cite[Corollary 3.18]{Hor21b} for the exactness of the localization. 
\end{proof}

\begin{PAR}
The factorization into trivial cofibration and fibration in $\mathcal{M}_{loc}$ is constructed by the small object argument. Every trivial cofibration is the retract of a transfinite composition of push-outs of trivial cofibrations in $\mathcal{M}$, and cofibrations of the form
\[ (\partial \Delta_{n} \hookrightarrow \Delta_n) \boxplus (U_\bullet \rightarrow Y')  \]
where $U_\bullet \rightarrow Y'$ is a cofibration obtained by factoring a \v{C}ech cover $U_\bullet \rightarrow Y$ as $U_\bullet \rightarrow Y' \rightarrow Y$ into cofibration followed by trivial fibration. See \cite[Proposition A.3.7.3]{Lur09}.
\end{PAR}

For later reference we state the following: 
\begin{LEMMA}\label{LEMMAFACTORCECH}Assume that $\mathcal{S}$ has finite limits. 
We may assume w.l.o.g.\@ that $Y'$ in the factorization $U_\bullet \rightarrow Y' \rightarrow Y$ of a \v{C}ech cover $U_\bullet \rightarrow Y$ into 
projective cofibration and projective trivial fibration is in the essential image of $\mathcal{S}^{\amalg, \Delta^{\op}}$. 
\end{LEMMA}
\begin{proof}
By Proposition~\ref{PROPSMALLS}
the Yoneda embedding $R: \mathcal{S}^{\amalg, \Delta^{\op}} \rightarrow \mathcal{SET}^{\mathcal{S}^{\op} \times \Delta^{\op}}$ preserves cofibrations and trivial fibrations. 
Therefore we may just factor the morphism in $\mathcal{S}^{\amalg, \Delta^{\op}}$. 
\end{proof}

\begin{PAR}
It would be desirable to have a more concrete description of $\mathcal{W}_{loc}$, like the one for the localization at all hypercovers. The fact that the localization is exact is equivalent to a different descriptions (see \cite{Hor21b}) in terms of local liftings w.r.t.\@ \v{C}ech hypercovers (i.e.\@ hypercovers which are also \v{C}ech weak equivalences). It is, however, self-referential and thus not really helpful in this respect. The description as smallest localizer in later sections gives a different characterization. 
\end{PAR}

\section{Localizers on diagram categories}

\begin{PAR}
Let $\mathcal{S}$ be a category.
After recalling the definition of the 2-category $\Dia(\mathcal{S})$ of diagrams in $\mathcal{S}$ and
the generalization of Grothendieck's notion of {\em ``localisateur fondamental''}, the main result of this section is that the  Grothendieck construction
\[ \int^{\amalg}: \mathcal{S}^{\amalg, \Delta^{\op}} \rightarrow \Dia(\mathcal{S}) \]
(cf.\@ \ref{PARINT}) maps weak equivalences between cofibrant objects\footnote{that restriction will later be repealed} to morphisms contained in any localizer provided $\mathcal{S}$ has finite limits. The reasoning is analogous to \cite{Cis04}. 
\end{PAR}

\begin{DEF} \label{DEFDIA}
Let $\mathcal{S}$ be a category. The 2-category $\Dia(\mathcal{S})$ of diagrams in $\mathcal{S}$ is the following 2-category:  
\begin{enumerate}
 \item The objects are the pairs $(I, F)$ with $I \in \Dia$ and $F \in \mathcal{S}^I$.
 \item The morphisms $(I, F) \rightarrow (J,G)$ are pairs $(\alpha, f)$ in which 
$\alpha: I \rightarrow J$ is a functor and $f: F  \rightarrow \alpha^* G$ a natural transformation.
\item The 2-morphisms $(\alpha, f) \Rightarrow (\beta, g)$ are 
 natural transformations $\mu: \alpha \Rightarrow \beta$ satisfying $\mu^*G \circ f= g$, where $\mu^*G$ is the induced natural transformation $\alpha^* G \rightarrow \beta^*G$.
\end{enumerate}
We call a morphism $(\alpha, f)$ of {\bf fixed shape} if $\alpha=\id$, and of {\bf diagram type} if $f$ consists of identities. 
Every morphism is obviously a composition of one of diagram type by one of fixed shape. 
\end{DEF}

\begin{PAR}Assume that $\mathcal{S}$ has fiber products. 
For two morphisms
\[
\xymatrix{ & \ar[d] (I, S) \\
(J, T) \ar[r] & (K, U)
  }
\]
we define a (non-commutative) fiber product
\[ (I, S) \times_{/(K, U)} (J, T) \]
in the obvious way, 
whose underlying diagram is the comma category $I \times_{/K} J$. 
\end{PAR}

\begin{PAR}\label{PARGCONSTR}
Given a pseudo-functor 
\begin{eqnarray*}
 F: A &\rightarrow& \Dia(\mathcal{S}) \\
  a & \mapsto & (I_a, S_a)
\end{eqnarray*}
 there is a Grothendieck construction
\[ \int F := (\int I, S) \]
where $I: A \rightarrow \Dia$ is the composition of $F$ with the forgetful functor and $S$ maps a pair $(a, i), a \in A, i \in I_a$ to $S_a(i) \in \mathcal{S}$. 
The functor $\int I \rightarrow A$ is an opfibration. 
\end{PAR}

\begin{DEF}\label{DEFWS}
A class of morphisms $\mathcal{W}$ in a category is called {\bf weakly saturated}, if it satisfies the following properties:
\begin{enumerate}
\item[(WS1)] Identities are in $\mathcal{W}$.
\item[(WS2)] $\mathcal{W}$ has the 2-out-of-3 property.
\item[(WS3)] If $p: Y \rightarrow X$ and $s: X \rightarrow Y$ are morphisms such that $p \circ s = \id_X$ and $s \circ p \in \mathcal{W}$ then $p \in \mathcal{W}$ (and hence 
$s \in \mathcal{W}$ by (WS2)). 
\end{enumerate}
\end{DEF}

The following definition was given in \cite[Definition 3.2.3]{Hor15} with a slightly different axiom (L4 left).

\begin{DEF}\label{DEFFUNLOC}
Let $\mathcal{S}$ be a category with finite limits. Assume we are given a Grothendieck pretopology on $\mathcal{S}$.
Consider the category $\Dia(\mathcal{S})$ of diagrams on $\mathcal{S}$.

A subclass $\mathcal{W}$ of 1-morphisms in $\Dia(\mathcal{S})$ is called a {\bf localizer} if
 the following properties are satisfied:
\begin{enumerate}
\item[(L1)]  $\mathcal{W}$ is weakly saturated.
\item[(L2 left)] If $D=(I,F) \in \Dia(\mathcal{S})$, and $I$ has a final object $e$, then the projection $D \rightarrow (e, F(e))$ is in $\mathcal{W}$.
\item[(L3 left)] If a commutative diagram in $\Dia(\mathcal{S})$
\[ \xymatrix{ (I, S) \ar[rd] \ar[rr]^{w} & & (J, T) \ar[dl]  \\
  & (K,U)   }\]
is such that for all $k \in K$ there is a covering $\{U_{i,k} \rightarrow U(k)\}$ such that the induced morphisms
\[  w \times_{/(K,U)} (k, U_{i,k}) \]
are in $\mathcal{W}$ for all $i$ then the morphism $w$ is in $\mathcal{W}$. 

\item [(L4 left)] The equivalent conditions (under assumption of (L1--L3)):
\item  If a functor $\alpha: I \rightarrow J$ and a morphism
$w: D_1=(I,\alpha^*T) \rightarrow D_2=(J,T)$ of pure diagram type is such that the morphisms 
\[  (j \times_{/J} I, \cdot) \rightarrow (\cdot, \cdot)  \]
are in $\mathcal{W}$ for all $j\in J$ then the morphism $w$ is in $\mathcal{W}$. 
\item If a fibration $\alpha: I \rightarrow J$ and a morphism
$w: D_1=(I,\alpha^*T) \rightarrow D_2=(J,T)$  of pure diagram type is such that the morphisms
\[ (I_j, \cdot)  \rightarrow (\cdot, \cdot)  \]
are in $\mathcal{W}$ for all $j\in J$ then the morphism $w$ is in $\mathcal{W}$. 
\end{enumerate}
\end{DEF}

See \cite[Proposition 3.2.12]{Hor15} for the proof of the equivalence of the two variants of (L4 left). Be aware that (L4 left) is stated in a slightly different form in \cite{Hor15} which does not
influence the proof of the equivalence.

There is an obvious dual notion of absolute {\bf colocalizer} in $\Dia^{\op}(\mathcal{S}) \cong \Dia(\mathcal{S}^{\op})^{2-\op}$ if we suppose that $\mathcal{S}^{\op}$ has a Grothendieck pretopology. 

One might want to include stronger saturatedness requirements into the definition. Actually it turns out (cf.\@ Theorem~\ref{SATZSAT} in the appendix) that
a localizer as defined above is automatically (strongly) saturated and, in particular, satisfies 2-out-of-6, is closed under retracts, etc.

In the following, we fix a localizer
$\mathcal{W}$ in $\Dia(\mathcal{S})$ and call the elements {\bf weak equivalences}.

We have the following immediate properties (cf.\@ \cite[Proposition 3.2.11]{Hor15}):

\begin{PROP}\label{PROPPROPERTIESLOCALIZER}
\begin{enumerate}
\item The localizer $\mathcal{W}$ is closed under coproducts. 

\item Let 
\[ \xymatrix{ I \ar@<3pt>[r]^s & \ar@<3pt>[l]^p J } \]
be an adjunction in $\Dia$ with $s$ right adjoint, let $S \in \mathcal{S}^J$  and consider the induced formal adjunction 
\[ \xymatrix{ (I,s^*S) \ar@<3pt>[rr]^-{\widetilde{s}=(s, \id)} & & \ar@<3pt>[ll]^-{\widetilde{p}=(p,u^*)} (J,S) } \]
where $u^*:  S \rightarrow p^*s^* S$ denotes the morphism induced by the unit. 
Then $\widetilde{s}$ and $\widetilde{p}$ are in $\mathcal{W}$. 
\item Consider a commutative diagram in $\Dia(\mathcal{S})$
\[ \xymatrix{ (I,S) \ar[rd] \ar[rr]^{w} & & (J,T) \ar[dl]  \\
  & (K,U)   }\]
  where the underlying vertical functors are opfibrations and the underlying functor of $w$ is a morphism of opfibrations. If for all $k \in K$ there are
  coverings $\{U_{k,i} \rightarrow U(k)\}$ such that the morphisms
  \[  w \times_{(K,U)} (k, U_{k,i}) \]
  are in $\mathcal{W}$ for all $i$ then $w$ is in $\mathcal{W}$. 

\item If $f: D_1 \rightarrow D_2$ is in $\mathcal{W}$, then $f \times E: D_1 \times E \rightarrow D_2 \times E$ is in $\mathcal{W}$  for any $E \in \Dia$.
\item Any morphism which is homotopic\footnote{Homotopy on $\Dia(\mathcal{S})$ is the equivalence relation generated by $\alpha \sim \beta$ if there is a 2-morphism $\mu: \alpha \Rightarrow \beta$.} to a morphism in $\mathcal{W}$ is in $\mathcal{W}$.
\end{enumerate}
\end{PROP}

\begin{KOR}\label{KORGCONSTRUCTION}
The Grothendieck construction $\int: \Dia(\mathcal{S})^I \rightarrow \Dia(\mathcal{S})$ (cf.\@ \ref{PARGCONSTR}) maps 
point-wise weak equivalences to weak equivalences. 
\end{KOR}
\begin{proof}
This follows directly from Proposition~\ref{PROPPROPERTIESLOCALIZER}, 3.\@ 
\end{proof}

\begin{PAR}\label{PARINT}
Applying the Grothendieck construction (cf.\@ \ref{PARGCONSTR}) to simplicial objects, we get functors
\begin{eqnarray*}
 \int^\amalg: \mathcal{S}^{\amalg, \Delta^{\op}} &\rightarrow& \Dia(\mathcal{S})  \\
 \int: \mathcal{S}^{\Delta^{\op}} &\rightarrow& \Dia(\mathcal{S}) 
  \end{eqnarray*}
In the first case $\mathcal{S}^{\amalg}$ is considered as a subcategory of $\Dia(\mathcal{S})$ consisting of diagrams whose underlying category is discrete (i.e.\@ a set) and in the second case $\mathcal{S}$ is considered as a subcategory of $\Dia(\mathcal{S})$ consisting of diagrams whose underlying category is the terminal category.
We equip $\mathcal{S}^{\Delta^{\op}}$, and $\mathcal{S}^{\amalg, \Delta^{\op}}$, with the split-projective model category structure of Theorem~\ref{SATZEXTENSIVEMODELCAT}, and Theorem~\ref{SATZSOMODEL}, respectively, if the necessary assumptions on $\mathcal{S}$ are fulfilled. {\bf We assume in any case that $\mathcal{S}$ has finite limits for the rest of the section. }

Obviously $\int$ is the composition
\[ \xymatrix{  \mathcal{S}^{\Delta^{\op}}  \ar[r]^-R&   \mathcal{S}^{\amalg, \Delta^{\op}}  \ar[r]^-{\int^{\amalg}} &  \Dia(\mathcal{S}), }  \]
where $R: \mathcal{S}^{\Delta^{\op}} \hookrightarrow \mathcal{S}^{\amalg, \Delta^{\op}}$ is the natural inclusion. 
Recall that weak equivalences in $\mathcal{S}^{\Delta^{\op}}$ and $\mathcal{S}^{\amalg, \Delta^{\op}}$ are the morphisms $f$ such that
\[ \Hom(X, f) \]
is a weak equivalence for all $\N$-small objects $X$ in $\mathcal{S}$, and in $\mathcal{S}^{\amalg}$, respectively --- in the latter case one can equivalently restrict to connected objects, i.e.\@ to the image of $\mathcal{S}$. 
\end{PAR}
\begin{PAR}\label{PARINTINT}Variant: For any $I \in \Dia$ we have a similar functor, denoted the same way, if $I$ is understood
\begin{eqnarray*}
 \int^\amalg: \mathcal{S}^{\amalg, I} &\rightarrow& \Dia(\mathcal{S})  
\end{eqnarray*}
where the underlying diagrams of the images come equipped with an opfibration to $I$. 
\end{PAR}

\begin{PAR}\label{PARPOINTWISE}
For $\mathcal{S} = \{ \cdot \}$ (terminal category) the first functor in \ref{PARINT} boils down to the usual Grothendieck construction
\[ \int: \mathcal{SET}^{\Delta^{\op}} \rightarrow \Dia.  \]
We also have a functor
\begin{eqnarray*} \Hom: \mathcal{S}^{\amalg} \times \Dia(\mathcal{S}) &\rightarrow& \Dia \\
 X,  (I, F:I \rightarrow \mathcal{S}) &\mapsto& \int \Hom(X, F(-)) 
 \end{eqnarray*}
 in which $\Hom(X, F(-)): I \rightarrow \mathcal{SET} \subset \Dia$ is considered as a functor with values in discrete categories. 
\end{PAR}
The following is a direct consequence of the definitions:
\begin{LEMMA}\label{LEMMAPOINTWISEINT}
We have for $X \in \mathcal{S}^{\amalg}$ and $S_\bullet \in \mathcal{S}^{\amalg, \Delta^{\op}}$: 
\[ \int \Hom_r(X, S_\bullet) \cong \Hom(X, \int^\amalg S_\bullet). \]
\end{LEMMA}

The main theorem of this section is the following:

\begin{SATZ}\label{THEOREMINTWECOF}
The functor $\int^\amalg$  of (\ref{PARINT}) maps weak equivalences between cofibrant objects to weak equivalences. 
\end{SATZ}

For this theorem the choice of Grothendieck pretopology is irrelevant and you may want to choose the trivial one. 
We will later prove that (cf.\@ Propositions~\ref{KORINTWEGENERAL}--\ref{KORINTWEGENERALCECH}), in fact, the restriction to cofibrant objects is unnecessary and that also $\int$ preserves weak equivalences, provided
that coproduct covers are in the chosen Grothendieck topology, i.e.\@ the latter is finer than the extensive topology.

\begin{PROP}\label{PROPPROPERTIESLOCALIZERII}
\begin{enumerate}
\item For a bisimplicial object $X_{\bullet, \bullet}$ in $\mathcal{S}^{\amalg, \Delta^{\op} \times \Delta^{\op}}$ the obvious morphism
\[ \int^{\amalg} \delta^* X_{\bullet, \bullet} \rightarrow  \int^{\amalg}  X_{\bullet, \bullet}   \]
is a weak equivalence. Here $\delta: \Delta^{\op} \rightarrow  \Delta^{\op} \times \Delta^{\op}$ is the diagonal. 
\item If $X_\bullet \rightarrow Y_\bullet$ is a morphism in $\mathcal{S}^{\amalg, \Delta^{\op}}$ such that $\int^{\amalg} X_\bullet \rightarrow \int^{\amalg} Y_\bullet$
is a weak equivalence then $\int^{\amalg} K \otimes X_\bullet \rightarrow \int^{\amalg} K \otimes Y_\bullet$ is a weak equivalence for any simplicial set $K$. 
\item The map $\int^{\amalg} \Delta_{n,\bullet} \otimes X_\bullet \rightarrow \int^{\amalg} X_\bullet$  is a weak equivalence for any $X_\bullet \in \mathcal{S}^{\amalg, \Delta^{\op}}$.  
\end{enumerate}
\end{PROP}
\begin{proof}
1. The morphism is of pure diagram type and by (L4 left) the assertion boils down to the contractibility of 
\[ (\Delta_n, \Delta_m, x)  \times_{/ \int X_{set}  }  \int \delta^* X_{set}  \]
where $X_{set}$ is the underlying bisimplicial set and $x \in X_{set,n,m}$. This diagram is isomorphic to
\[ \int_{\Delta^{\op}} \Delta_{n, \bullet} \otimes \Delta_{m, \bullet}  \]
which is contractible by \cite[Corollary~3.3.4]{Hor15}.

2. For a {\em set} $K$, we have
\[  \int^{\amalg} K \otimes X_\bullet = K \times \int^{\amalg} X_\bullet \]
hence the assertion follows from Proposition~\ref{PROPPROPERTIESLOCALIZER}, 4.\@ For a simplicial set $K$ consider the morphism of bisimplicial objects
\[ K_n \otimes X_m \rightarrow K_n \otimes Y_m. \]
Its horizontal simplicial objects are thus mapped to weak equivalences under $\int^{\amalg}$. 
From  Proposition~\ref{PROPPROPERTIESLOCALIZER}, 3.\@ applied to the diagram
\[ \xymatrix{
\int^{\amalg} X_{\bullet, \bullet} \ar[rd] \ar[rr] & &  \int^{\amalg} Y_{\bullet, \bullet} \ar[ld] \\
& (\Delta^{\op}, \cdot) 
}\]
in which the projections are opfibrations,
follows that a morphism of bisimplicial objects is mapped to a weak equivalence under $\int^{\amalg}$ provided that the morphisms between its horizontal simplicial objects are mapped to weak equivalences under $\int^{\amalg}$ . The assertion follows therefore from 1.

3. Consider the pullback 
\[ \xymatrix{ \pi^*X_\bullet: \Delta_n \otimes_{/\Delta^{\op}} \Delta^{\op} \ar[r]^-\pi & \Delta^{\op} \ar[r]^{X_\bullet} &  \mathcal{S}^{\amalg}. } \]
We have $\int^{\amalg} \Delta_{n, \bullet} \otimes X_\bullet = \int^{\amalg} \pi^* X_\bullet$. The morphism $\int^{\amalg} \pi^*X_\bullet  \rightarrow \int^{\amalg}  X_\bullet $ is of pure diagram type. 
By (L4 left) we have to show that for $x \in X_{set, m}$ the diagram
\[ \{ (\Delta_m, x) \} \otimes_{/\int X_{set}} \int  \pi^* X_{set}  \]
is contractible\footnote{We call a diagram $I$ contractible if the morphism $(I, \cdot) \rightarrow (\cdot, \cdot)$ is in the localizer.} where $X_{set}$ is the image of $X_\bullet$ under the forgetful functor $\mathcal{S}^{\amalg, \Delta^{\op}} \rightarrow \mathcal{SET}^{\Delta^{\op}}$. This diagram is isomorphic to 
\[ \{ \Delta_m \} \otimes_{/\Delta^{\op}} \int \Delta_{n,\bullet}  = \int_{\Delta^{\op}} \Delta_{n, \bullet} \otimes \Delta_{m, \bullet}  \]
which is contractible by \cite[Corollary~3.3.4]{Hor15}.
\end{proof}

\begin{KOR}\label{KORLEFTHOMOTOPYEQUIVALENCE}
Every left homotopy equivalence in $\mathcal{S}^{\amalg, \Delta^{\op}}$ maps to a weak equivalence under $\int^\amalg$. 
In particular, trivial fibrations {\em between cofibrant objects} map to weak equivalences. 
\end{KOR}

\begin{LEMMA}\label{LEMMAPUSHOUTTRANSFINITE}
\begin{enumerate}
\item
Let $S: \lefthalfcap \rightarrow \mathcal{S}^{\amalg}$ a diagram in which one of the morphisms is of the form $A \rightarrow A \amalg B$. Then
\[ \int^\amalg S \rightarrow \colim S\]
is a weak equivalence. 
\item
Let $O$ be an ordinal and let $S: O \rightarrow \mathcal{S}^{\amalg}$ be a diagram in which all morphisms are of the form $A \rightarrow A \amalg B$. Then
\[ \int^\amalg S \rightarrow \colim S\]
is a weak equivalence. 
\end{enumerate}
\end{LEMMA}
\begin{proof}
1.\@ The diagram
\[
\xymatrix{
A \ar[d] \ar[r] & C \ar[d] \\
A \amalg B \ar[r] & C \amalg B = \colim(S) 
}
\]
in $\mathcal{S}^{\amalg}$ is Cartesian. 
Let $(i, S_i) \hookrightarrow \colim S$ be a connected component. It is either in $C$ or in $B$.
We have to show that $\int S \times_{/(i, S_i)} (i, S_i) \rightarrow (i, S_i)$ is a weak equivalence.
This is either the diagram $(\cdot, B_i)$ over $(\cdot, B_i)$ or the diagram 
\[
\xymatrix{
A_i \ar[d] \ar[r] & C_i \\
A_i &
}
\]
over $(\cdot, C_i)$. This might be factored as
\[
\left( \vcenter{ \xymatrix{
A_i \ar[d] \ar[r] & C_i \\
A_i &
} } \right)  \rightarrow \left(\xymatrix{
A_i \ar[r] & C_i \\
} \right) \rightarrow (C_i).
\]
The corresponding morphisms are weak equivalences by (L4 left) and (L2 left).

Let $(i, S_i) \rightarrow \colim S$ be a connected component. Then $\int^\amalg S \times_{/(i, S_i)} (i, S_i) \rightarrow (i, S_i)$
is the integration of the constant diagram $O_{o} \rightarrow \mathcal{S}$ with value $S_i$, where $O_o$ is the subset of elements $> o$ and $o \in O$ is such that $o \rightarrow o+1$
is a mapped to a pushout of the form  $A \rightarrow A \amalg B$ and $S_i$ occurs in $B$. 
Constant diagrams of ordinal shape are contractible (e.g.\@ because they have an initial object). 
\end{proof}

\begin{PROP}\label{PROPPUSHOUT}
If a cofibration $f$ maps to a weak equivalence under $\int^\amalg$ then also any pushout of $f$ maps to a weak equivalence. 
\end{PROP}
\begin{proof}
Consider a diagram 
\[
\xymatrix{ G \ar[r]^f  \ar[d] & F \ar[d] \\
L \ar[r] & H 
  } \]
in $\mathcal{S}^{\amalg, \Delta^{\op}}$ in which $H: \Delta^{\op} \rightarrow \mathcal{S}^{\amalg, \Delta^{\op}}$ is the push-out. The functors $G, F, L$ assemble to a functor  $X: \Delta^{\op} \times \lefthalfcap \rightarrow \mathcal{S}^{\amalg}$.
First, we claim that 
\[ \int_{\lefthalfcap \times \Delta^{\op}}^{\amalg} X \rightarrow \int^{\amalg}H  \]
is a weak equivalence. It suffices to see this point-wise in $\Delta^{\op}$, i.e.\@ we have to prove that
\[ \int_{\lefthalfcap} X_n \rightarrow H_n  \]
is a weak equivalence. This follows from Lemma~\ref{LEMMAPUSHOUTTRANSFINITE}, 1.\@ because a cofibration is in particular degree-wise of the form $A \mapsto A \amalg B$ (i.e.\@ in $\mathcal{L}_{\proj,\spl}$). Hence for two push-out diagrams $(i \in \{0, 1\})$ in $\mathcal{S}^{\amalg, \Delta^{\op}}$ 
\[
\xymatrix{ G_i \ar[r]^{f_i}  \ar[d] & F_i \ar[d] \\
L_i \ar[r] & H_i
  } \]
such that  $f_i$ is a cofibration and a morphism between them such that 
\begin{equation} \label{eq1}
\int^{\amalg } F_0 \rightarrow F_1,  \int^{\amalg } G_0 \rightarrow G_1, \text{ and } \int^{\amalg } L_0 \rightarrow L_1, \text{ are weak equivalences} 
\end{equation}
 then also
\[ \int^{\amalg } H_0 \rightarrow \int^{\amalg } H_1 \]
is a weak equivalence. Finally, consider the morphism of push-out diagrams 
\[
\left(\vcenter{\xymatrix{ F \ar[r]^\id  \ar[d] & F \ar[d] \\
L \ar[r] & L
  }}\right) \rightarrow
  \left(\vcenter{\xymatrix{ F \ar[r]^f  \ar[d] & G \ar[d] \\
L \ar[r] & H
  }}\right). \]
which satisfies (\ref{eq1}) because $f$ is a weak equivalence. 
Hence $L \rightarrow H$ is a weak equivalence.  
\end{proof}

\begin{PAR}Given $e \in \{0, 1\}$, $n \in \N$, and $S \in \mathcal{S}$, 
consider  the monomorphisms 
\[ \underbrace{(\Delta_n \times \{e\} \cup  \partial \Delta_n \times \Delta_1)}_{=:\Lambda_e( \Delta_n \times \Delta_1)} \otimes S \rightarrow (\Delta_n \times \Delta_1) \otimes S  \]
coming from the diagram: 
\[
\xymatrix{ (\partial \Delta_n \times \{e\}) \otimes S  \ar[rr]^{\id \times \delta_{1}^{1-e}} \ar[d] & & ( \partial \Delta_n \times \Delta_1)  \otimes S \ar[d] \\ 
( \Delta_n \times \{e\} )  \otimes S \ar[rr]^{\id \times \delta_{1}^{1-e}}  & & ( \Delta_n \times \Delta_1)  \otimes S
  } \]
\end{PAR}

\begin{PROP}
The morphisms $\Lambda_e( \Delta_n \times \Delta_1) \otimes S \rightarrow (\Delta_n \times \Delta_1) \otimes S$  map to weak equivalences under $\int^\amalg$. 
\end{PROP}

\begin{proof}
By Proposition~\ref{PROPPROPERTIESLOCALIZERII}, 2.\@\footnote{Note  that e.g.\@ $(\partial \Delta_n \times \{e\}) \otimes S \rightarrow ( \partial \Delta_n \times \Delta_1)  \otimes S$ is the same as
$\partial \Delta_n \otimes (\{e\} \otimes S) \rightarrow \partial \Delta_n \otimes (\Delta_1 \otimes S)$.} the morphisms $(\partial \Delta_n \times \{e\}) \otimes S \rightarrow ( \partial \Delta_n \times \Delta_1)  \otimes S$ and
$( \Delta_n \times \{e\} )  \otimes S \rightarrow ( \Delta_n \times \Delta_1)  \otimes S$ are  mapped to weak equivalences. Thus by Proposition~\ref{PROPPUSHOUT} also the push-out
\[   (\Delta_n \times \{ e \}) \otimes S \rightarrow (\Delta_n \times \{e\} \cup  \partial \Delta_n \times \Delta_1) \otimes S \]
is mapped to a weak equivalence and so the same holds for the given morphism by 2-out-of-3. 
\end{proof}

\begin{PROP}\label{PROPTRANSFINITE}
Let $O$ be an ordinal. 
If a sequence $O \rightarrow \mathcal{S}^{\amalg, \Delta^{\op}}$ of cofibrations maps to weak equivalences under $\int^\amalg$ then also its transfinite composition maps to a weak equivalence.
\end{PROP}
\begin{proof}
If a functor $X: O \times \Delta^{\op} \rightarrow \mathcal{S}^{\amalg}$ is a cofibration for all morphisms in $O$ then the morphism
\[ \int_{O \times \Delta^{\op}} X \rightarrow \int_{\Delta^{\op}} \colim X   \]
is a weak equivalence by Lemma~\ref{LEMMAPUSHOUTTRANSFINITE}, 2.\@ because this can be checked degree-wise. It follows that, given two diagrams $X_0, X_1: O \times \Delta^{\op} \rightarrow \mathcal{S}^{\amalg}$ mapping all morphisms on $O$ to cofibrations, 
and a point-wise weak equivalence between them, also $\int^{\amalg} \colim X_0  \rightarrow \int^{\amalg} \colim X_1$ is a weak equivalence. 
Applying this to  $X_1 := X$ and $X_0 := X(0)$ (constant in $O$) the result follows. 
\end{proof}

\begin{proof}[Proof of Theorem~\ref{THEOREMINTWECOF}.]
Let $f$ be a weak equivalence between cofibrant objects.
It can be factored as
\[ f = p \circ \iota \]
where $\iota$ is a transfinite composition of pushouts of morphisms of the form\footnote{By \cite[Chapter IV]{GZ67} trivial cofibrations of simplicial sets are generated by the set $\Lambda_e( \Delta_n \times \Delta_1) \rightarrow \Delta_n \times \Delta_1$ and thus trivial cofibrations in $\mathcal{S}^{\amalg, \Delta^{\op}}$ are generated by the morphisms in question.}
\[ \Lambda_e( \Delta_n \times \Delta_1) \otimes S \rightarrow (\Delta_n \times \Delta_1) \otimes S  \] for $S \in \mathcal{S}$, and where $p$ is a trivial fibration between cofibrant objects.
The trivial fibration is mapped to a weak equivalence by Corollary~\ref{KORLEFTHOMOTOPYEQUIVALENCE}.
Therefore the statement follows from Propositions~\ref{PROPPUSHOUT}--\ref{PROPTRANSFINITE}.
\end{proof}

\section{Localizers and (co)homological descent}

\begin{PAR}
The notion of localizer is well-suited to study questions of (co)homological descent in derivators. Recall \cite{Hor15} the notion of fibered derivator as well as \cite[3.5]{Hor15} the notion of weak $\DD$-equivalence for a fibered derivator $\DD \rightarrow \SSS$ (cf.\@ also Definition~\ref{DEFWEAKDEQUIV}). This is a generalization of a notion of Cisinski~\cite{Cis08}. (The notion of strong $\DD$-equivalence will play a role later when extending fibered derivators to stacks in Section~\ref{SECTEX}.) 
In this section let $\mathcal{S}$ be a category with finite limits and
Grothendieck pretopology and denote by $\SSS$ the prederivator represented by $\mathcal{S}$. The definitions make sense more generally for a fibered derivator over a right derivator $\SSS$ --- we refer to \cite{Hor15} for details. 
\end{PAR}

\begin{PAR}
Let $\DD \rightarrow \SSS$ be a left fibered derivator satisfying (FDer0 right) as well. Recall \cite[2.6]{Hor15} that $\DD$ induces a 2-pseudo-functor $\Dia(\mathcal{S})^{\op} \rightarrow \mathcal{CAT}$ mapping $(I, T) \mapsto \DD(I)_T$ and a morphism $\mu: (I, T) \rightarrow (J, U)$ given by $\alpha: I \rightarrow J$ and $f: T \rightarrow \alpha^* U$ to $\mu^* := f^\bullet \alpha^*$ where $f^\bullet$ is a fixed pull-back functor along $f$.  From the axioms of a left fibered derivator together with (FDer0 right) it follows that these functors have left adjoints, namely  $\alpha_!^{(U)} f_\bullet$. 
Sometimes this pseudo-functor is taken to be the basic datum as for instance in Ayoub's notion of algebraic derivator. 
\end{PAR}

\begin{DEF}\label{DEFWEAKDEQUIV}Let $S$ be an object in $\mathcal{S}$.
A morphism 
\[ \xymatrix{ 
D_1 \ar[rr]^\alpha \ar[rd]_{\pi_1} & & D_2 \ar[ld]^{\pi_2} \\
& S
 } \]
in $\Dia(\mathcal{S})/S$ is  a {\bf $\DD$-equivalence} over $S$ if the morphism
\[ \pi_{1,!} \pi_{1}^* \rightarrow  \pi_{2,!} \pi_{2}^*  \]
induced by $\alpha$ is an isomorphism.
\end{DEF}

\begin{LEMMA}\label{LEMMADDEQUIVALENCE}
\begin{enumerate}
\item If $\alpha: D_1 \rightarrow D_2$ and $\beta: D_1 \rightarrow D_2$ are morphisms over $S$ and are connected by a 2-morphism $\mu: \alpha \Rightarrow \beta$
compatible with the morphisms to $S$  then $\alpha$ is a $\DD$-equivalence over $S$ if and
only if $\beta$ is a $\DD$-equivalence over $S$.
\item Every morphism in a formal adjunction in the 2-category $\Dia(\mathcal{S})/S$ is a $\DD$-equivalence over $S$. 
\item If $I$ contains a final object $i$ then $(I, S) \leftrightarrow (\cdot, S(i))$ are $\DD$-equivalences over $S(i)$. 
\end{enumerate}
\end{LEMMA}
\begin{proof}
1.\@ follows directly from the 2-functoriality, and
2.\@ is an immediate consequence of 1. For
3.\@ note that $(I, S) \leftrightarrow (\cdot, S(i))$ is a formal adjunction in the 2-category $\Dia(\mathcal{S})$
\end{proof}

The following is a slight generalization of the Main Theorem of weak homological descent \cite[Theorem 3.5.5]{Hor15}. 
First recall: 
\begin{DEF}\label{DEFLOCAL}
Let $p: \DD \rightarrow \SSS$ be a left fibered derivator satisfying also (FDer0 right). 
A morphism $f: U \rightarrow S$ in $\mathcal{S}$ is {\bf $\DD$-local} if

\begin{itemize}
\item[(Dloc1 left)] 
The morphism $f$ satisfies {\bf base change}: for any diagram $Q \in \DD(\Box)$  with underlying diagram
\[ \xymatrix{A \ar[r]^-{\widetilde{F}} \ar[d]_-{\widetilde{G}} & B \ar[d]^{\widetilde{g}} \\ C \ar[r]_{\widetilde{f}} & D } \]
such that $p(Q)$ in $\mathcal{S}^\Box$ is a pull-back-diagram with $p(\widetilde{f})=f$ the following holds true: 
If $\widetilde{F}$ and ${\widetilde{f}}$ are Cartesian, and $\widetilde{g}$ is coCartesian then also $\widetilde{G}$ is 
coCartesian.\footnote{In other words, if 
\[ \xymatrix{ U \times_S V   \ar[r]^-F \ar[d]_-G & V \ar[d]^-g \\ U \ar[r]_-f & S } \]
is the underlying diagram of $p(Q)$ then the exchange morphism
\[ G_\bullet F^\bullet \rightarrow f^\bullet g_\bullet  \]
is an isomorphism.}

\item[(Dloc2 left)] The morphism of left derivators 
$f^\bullet: \DD_S \rightarrow \DD_U$
commutes with homotopy colimits.
\end{itemize}

The left fibered derivator $p: \DD \rightarrow \SSS$ is {\bf local} w.r.t.\@ the pretopology on $\mathcal{S}$, if
the following conditions hold: 
\begin{enumerate}
\item Every morphism $U_i \rightarrow S$ which is part of a cover is $\DD$-local.
\item For a cover $\{f_i: U_i \rightarrow S\}$ the family
$\{ (f_i)^\bullet: \DD(S) \rightarrow \DD(U_i)\} $
is jointly conservative.
\end{enumerate}
\end{DEF}

\begin{SATZ}\label{SATZWEAKHOMDESCENT}
Let $p: \DD \rightarrow \SSS$ be a left fibered derivator which is local w.r.t.\@ the Grothendieck pretopology on $\mathcal{S}$, and such that all push-forward functors $f_\bullet$ are conservative. Then
all $\mathcal{W}_S$ are localizers. 
\end{SATZ} 
There exists an obvious dual variant (cohomological descent) whose formulation we leave to the reader. Without the assumption regarding conservativity, the $\mathcal{W}_S$ only form a
{\em system of relative localizers}, a notion that we will not need in this article. For the statement of this Theorem to be true we had to slightly change property (L4 left) with respect to \cite{Hor15}. 
The Main Theorem of weak homological descent was proven in \cite{Hor15} under the assumption that $\DD$ is also a right fibered derivator. The assumption, however, is not needed at least for this formulation of the Theorem. 

\begin{proof}
(L1) is clear. (L2 left) is Lemma~\ref{LEMMADDEQUIVALENCE}, 3. 
For (L3 left) consider a diagram
\[ \xymatrix{ D_1 \ar[rddd]_{\pi_1} \ar[rd]\ar[rr]^f & & D_2 \ar[dl] \ar[dddl]^{\pi_2} \\
& D_3 = (E, F) \ar[dd]|{\pi_3} \\
\\
 & S
 }\]
  in $\Dia(\mathcal{S})/S$. For all $e \in E$ let $\{ U_{e,i} \rightarrow F(e) \}$ be a covering and 
 assume that
 \[ D_{1,i,e} := D_1 \times_{/D_3} U_{e,i} \rightarrow D_{2,i,e} := D_2 \times_{/D_3} U_{e,i} \]
 is in $\mathcal{W}_S$ for all $e \in E$ and for all $i$, i.e.\@ that 
 \[ \pi_{D_{1,i,e},!} \pi_{D_{1,i,e}}^* \rightarrow \pi_{D_{2,i,e},!} \pi_{D_{2,i,e}}^* \]
 is an isomorphism. 

 Consider the following diagrams for $k=0,1$:
 \[\xymatrix{ 
 D_{k,e,i} \ar[r]^{B_{k,i,e}} \ar[d]_{\Pi_i} &  D_{k} \ar[rdd]^{\pi_i} \ar[d]_{p_k} \\
 U_{i,e} \ar[rrd]_{\pi_{U_{i,e}}} \ar[r]^{\beta_{i,e}} & D_3 \ar[rd]|{\pi_3} \\
 & & S}\]
 By conservativity of $\pi_{U_{i,e},!} = \pi_{U_{i,e},\bullet}$ (the assumption) the morphism 
 \[ \Pi_{1,!} \pi_{D_{1,i,e}}^* \rightarrow \Pi_{2,!} \pi_{D_{2,i,e}}^* \]
is  an isomorphism.
By base change (DLoc1 left) we have that 
 $\Pi_{k,!} B_{k,i,e}^*   \cong  \beta_{i,e}^* p_{k,!}$ (cf.\@ \cite[Proposition 2.6.8.2]{Hor15}). 
Therefore 
 \[ \beta_{i,e}^*  p_{1,!} \pi_{1}^* \rightarrow \beta_{i,e}^*  p_{2,!} \pi_{2}^* \]
 is an isomorphism. Since $\DD$ is local, and by (Der2), the collection of $\beta_{i,e}^*$ for all $i$ and $e$ is conservative, hence
 \[   p_{1,!} \pi_{1}^* \rightarrow  p_{2,!} \pi_{2}^* \]
 is an isomorphism and therefore also
 \[   \pi_{1,!} \pi_{1}^* \rightarrow  \pi_{2,!} \pi_{2}^*. \]

For (L4 left) 
let 
\[ \xymatrix{ 
D_1:=(I, \alpha^* T) \ar[rr] \ar[rd]_{\pi_1} & &  D_2:=(J, T) \ar[ld]^{\pi_2} \\
& S
} \]
be a morphism in $\Dia(\mathcal{S})/S$ of pure diagram type 
such that  $\alpha: I \rightarrow J$ is a fibration.
 Let $f: T \rightarrow \pi_J^*S$ be the structural morphism.  
By assumption 
\[ \pi_{I_j,!}  \pi_{I_j}^* \rightarrow \id  \]
is an isomorphism in $\DD_{S}(\cdot)$ for all $j$. 
Therefore Theorem~\ref{FAIBLE} below applies and
\[ \pi_{J,!}  \alpha_!  \alpha^*  \rightarrow \pi_{J,!} \]
is an isomorphism on $\DD_{S}(\cdot)$. Therefore also
\[ \pi_{J,!}  \alpha_!  \alpha^* f_\bullet f^\bullet \pi_J^*  \rightarrow \pi_{J,!} f_\bullet f^\bullet \pi_J^* \]
is an isomorphism, which is the same as
\[ \pi_{I,!} (\alpha^*f)_\bullet (\alpha^*f)^\bullet   \pi_I^*  \rightarrow \pi_{J,!} f_\bullet f^\bullet \pi_J^*, \]
or equivalently
\[ \pi_{1,!} \pi_1^* \rightarrow  \pi_{2,!} \pi_2^*.   \]
\end{proof}

We have already made use of the following theorem:

\begin{SATZ}\label{FAIBLE}
Let $\DD$ be a left derivator. If $\alpha: I \rightarrow J$ is a fibration such that for each fibre $p_{I_j}: I_j \rightarrow j$ the counit $p_{I_j, !} p_{I_j}^* \rightarrow \id$ is an isomorphism then 
the canonical morphism
\[ p_{I,!} \alpha^* \rightarrow p_{J,!}  \]
is an isomorphism. 
\end{SATZ}
In the language of \cite[3.11]{Cis08} this means that every ``d\'erivateur faible \`a gauche'' is a ''d\'erivateur \`a gauche'', a fact stated in \cite{Cis08} without proof.
We need a kind of adjoint of axiom (Der2):

\begin{LEMMA}\label{LEMMAADJOINTDER2}
Let $\DD, \EE$ be left derivators. Let $I$ be a diagram and let $\DD_{I}$ be the fibre \cite[Theorem 1.30]{Gro13} of $\DD$ above $I$. If
\[ F, G: \DD_{I} \rightarrow \EE, \]
 are two continuous (i.e.\@ commuting with homotopy colimits) morphisms of left derivators and
\[ \mu: F \Rightarrow G \] is a 2-morphism such that $\mu(\cdot)$ is an isomorphism 
on all objects of the form $i_!  \mathcal{E}$ for $i \in  I$  and $\mathcal{E} \in \DD(\cdot)$ then 
$\mu$ is an isomorphism. 
\end{LEMMA}
Observe that the Lemma follows immediately from (Der2) in case that $F$ and $G$ have right adjoints. 
\begin{proof}Consider the diagram $\twc I$ (cf.\@ \cite[7.3]{Hor16}). Its objects are sequences $i_1 \rightarrow i_2 \rightarrow i_3$ in $I$ and its morphims are commutative diagrams
\[ \xymatrix{ i_1 \ar[r]\ar[d] &  i_2 \ar[r] &  i_3 \ar[d] \\
 i_1' \ar[r] &  i_2' \ar[r] \ar[u] &  i_3'. } \]
It comes equipped with two obvious functors
\[ \xymatrix{ I & \ar[l]_{\pi_3}\twc I \ar[r]^-{\pi_1} & I   } \]
and a natural transformation $\mu: \pi_1 \Rightarrow \pi_3$ which induces a natural transformation
\begin{equation}\label{eq7}
 \pi_{3,!}  \pi_1^* \Rightarrow \id.
\end{equation}
Let us assume for the moment that (\ref{eq7}) is an isomorphism.  
The functor $\pi_{3,!}$ is the composition:
\[ \pi_{3,!} = \pi_{3,!}' p_! \]
where $p$ and $\pi_3'$ are the opfibrations  
\[ \xymatrix{ \twc I \ar[r]^-p & \tw I \times  I \ar[r]^-{\pi_3'} \ar[r] & I. } \]
Since $p$ has discrete fibers, 
the functor $p_! \pi_1^*$ is point-wise at $(\nu: i \rightarrow j, k)$ given by
\[  \coprod_{\alpha \in \Hom(j, k)}  i^* \mathcal{E}  =  k^* j_!  i^* \mathcal{E} . \] 
Thus every object $\mathcal{E}$ in the fiber $\DD_I(\cdot)$ is of the form 
$p_{\tw I,!} \mathcal{E}'$
where $p_{\tw I,!}$ is the homotopy colimit functor {\em of the fiber},
with $\mathcal{E}' \in \DD_I(\tw I)$ which is point-wise (in $\tw I$) of the form $j_! \mathcal{F}$ for some $\mathcal{F} \in \DD(\cdot)$.
The statement follows. 

It remains to show that (\ref{eq7}) is an isomorphism. This can be checked point-wise in $I$, i.e.\@ we are left to show that 
 the morphism
 \[ i^* \pi_{3,!} \pi_1^* \Rightarrow  i^* \]
 is an isomorphism. Consider the following functors:
\[ \xymatrix{ \tw(I \times_{/I} i) \ar[r]^-{\pi_{13}'} \ar[d]^{\pi_3''} & I \times_{/I} i  \ar[r]^-{\pi_1'} & I   \\
 \{i\}    } \]
 Since $\pi_3$ is an opfibration we have
\begin{equation*}   i^* \pi_{3,!} \pi_1^* \cong  p_{\tw(I \times_{/I} i),!} \pi_1^* \cong
   p_{\tw(I \times_{/I} i),!} (\pi_{13}')^* (\pi_1')^*. 
\end{equation*}
Thus it suffices to see that the counit 
 \[ (\pi_{13}')_! (\pi_{13}')^* \Rightarrow \id \]
is an isomorphism. (Note that $p_{I \times_{/I} i, !}(\pi_1')^* \cong i^*$ is an isomorphism.) However $\pi_{13}'$ has an obvious left adjoint $\iota_{13}$ such that the unit $\id \Rightarrow \pi_{13}' \iota_{13}$ is an isomorphism.
Therefore, by 2-functoriality of the prederivator, $(\pi_{13}')^*$ has the right adjoint  $\pi_{13,*}' := \iota_{13}^*$ such that the unit $\id \Rightarrow\pi_{13,*}' (\pi_{13}')^*$ is an isomorphism. 
However, the above counit is just the adjoint of this isomorphism and thus also an isomorphism. 
\end{proof}

\begin{proof}[Proof of Theorem~\ref{FAIBLE}]
Domain and codomain of the natural transformation in question give rise to functors from the fiber above $J$ 
\[ \DD_J \rightarrow \DD \]
which are continuous, i.e.\@ commute with homotopy colimits. Hence, by Lemma~\ref{LEMMAADJOINTDER2}, it suffices to show the assertion on objects of the form $j_! \mathcal{E}$ for $\mathcal{E} \in \DD(\cdot)$ and $j \in J$, i.e.\@ we have to show that 
\[ p_{I,!} \alpha^* j_! \rightarrow p_{J,!} j_!  = \id \]
is an isomorphism for all $j \in J$. Since $\alpha$ is a fibration, the diagram
\[ \xymatrix{ I_j \ar[r]^{p_{I_j}} \ar[d]_{\iota} & \cdot \ar[d]^j \\
 I \ar[r]_{\alpha} & J  } \]
 is homotopy exact. Therefore we have to show that 
\[ p_{I_j,!} p_{I_j}^* = p_{I,!} \iota_! p_{I_j}^* \rightarrow \id \]
is an isomorphism which is true by assumption.
\end{proof}

\section{The fibered derivators of higher stacks}

In Section~\ref{SECTMC} we have seen several model categories presenting the homotopy theory of higher stacks.
In this section we prove that they all yield fibered derivators over the base category $\mathcal{S}$ with an explicit Bousfield-Kan formula for
relative homotopy Kan extensions. 

\begin{PAR}\label{PARMC}Let $\mathcal{S}$ be a category with finite limits. 
Recall from Section~\ref{SECTMC} that we have the following simplicial model categories:
\begin{enumerate}
\item  $\mathcal{S}^{\Delta^{\op}}$, if $\mathcal{S}$ is extensive (thus big in general), has finite limits, and every object is a coproduct of $\N$-small objects, with the
split-projective structure,
\item  $\mathcal{S}^{\amalg, \Delta^{\op}}$, if $\mathcal{S}$ has finite limits (and arbitrary size), with the
split-projective structure;
\item  $\mathcal{SET}^{\mathcal{S}^{\op} \times \Delta^{\op}}$, if $\mathcal{S}$ is small, with the
projective model structure.
\item  $\mathcal{SET}^{\mathcal{S}^{\op} \times \Delta^{\op}}_{loc}$, if $\mathcal{S}$ is small, with the left Bousfield localization at the \v{C}ech covers of the model structure in 3. 
\end{enumerate}
The goal of this section is to show that the above model categories give rise to fibered derivators in the sense of \cite[2.3.6]{Hor15} over $\SSS$, the prederivator represented by $\mathcal{S}$. 
\end{PAR}

\begin{PROP}
Let $I$ be a small category and let $\mathcal{M}$ be one of the model categories in \ref{PARMC}. Then the category
$\mathcal{M}^I$
of $I$-shaped diagrams can be equipped with the projective extension of the previous model category structure, i.e.\@ fibrations and weak equivalences are those which are point-wise of this form. 
\end{PROP}
\begin{proof}
This is well-known for the projective model category structure on simplicial presheaves (and its localization) and proven in \cite[Theorem~6.1]{Hor21} for the split-projective structure. 
\end{proof}

We have in addition: 
\begin{SATZ}\label{SATZBIFIBMC}
Let $I$ be a small category and let $\mathcal{M}$ be one of the model categories in \ref{PARMC}. Then the functor\footnote{$\mathcal{M}^I / \mathcal{S}^I$ is the usual shorthand for the comma category $\mathcal{M}^I \times_{/\mathcal{S}^I} \mathcal{S}^I$ }
\[ \mathcal{M}^I / \mathcal{S}^I \rightarrow \mathcal{S}^I \]
equipped point-wise with the model category structure of over-category, is a bifibration of model categories. 
Furthermore the push-forward functors $f_\bullet$ and pull-back functors $f^\bullet$ preserve weak equivalence (i.e.\@ are equal to their derived functors). 
\end{SATZ}
\begin{proof}
The functor is just (a restriction of) the canonical over-category bifibration. The assertion is that for all morphisms $f: S \rightarrow T$ in $\mathcal{S}$ the push-forward and pull-back functors $f_\bullet$ and $f^\bullet$ are Quillen adjunctions. In this case the left adjoint $f_\bullet$  is the functor that composes the structural morphism with $f$. It preserves all three classes by definition of the model category structure on the over-category. Together with the right adjoint $f^\bullet(X) = X \times_T S$ it therefore forms a Quillen adjunction. 
This is true for any model category. In these cases, however, one easily checks that $f^\bullet$ preserves weak equivalences. For case 4.\@ see e.g.\@ \cite[Proposition~3.17]{Hor21b}.
\end{proof}

\begin{KOR}\label{KORBIFIB}
Let $\mathcal{M}$ be one of the model categories in \ref{PARMC}, and let $I$ be a diagram. Then the functor
\[ \DD_{\mathcal{M}}(I):= \mathcal{M}^I / \mathcal{S}^I [\mathcal{W}_I^{-1}] \rightarrow \mathcal{S}^I \]
where $\mathcal{W}_I$ is the union of the weak equivalences in the fibers,
is a bifibration whose pull-back and push-forward functors are computed point-wise (i.e.\@ commute with pull-backs along diagrams). 
\end{KOR}
\begin{proof}By \cite[Proposition 5.1.9]{Hor15} the functor is a bifibration and the new pull-back and push-forward functors are the derived functors of the pull-back and push-forward functors in the original bifibration of model categories. We conclude because those are already derived (i.e.\@ preserve weak equivalences) and computed point-wise. 
\end{proof}

\begin{SATZ}\label{SATZHOCOLIM}
Let $\mathcal{M}$ be one of the model categories in \ref{PARMC}. Then the association
\[ I \mapsto \DD_{\mathcal{M}}(I) \]
defines a left fibered derivator over $\SSS$ (the represented prederivator associated with $\mathcal{S}$) with domain $\Dia$ (all small categories) and a right fibered derivator on homotopically finite diagrams. In case 3.\@ and 4.\@ it is a right fibered derivator on $\Dia$. 
On point-wise cofibrant objects  $X \in \mathcal{M}^I$  the homotopy colimit  is given explicitly by
\[ \hocolim X = \int^{I} N(- \times_{/I} I) \otimes X.  \]
This particular coend is computed by means of coproducts, more precisely, we have
\begin{equation}\label{eqhocolim} (\hocolim X)_n = \widetilde{X}_{n,n}, \end{equation}
where $\widetilde{X}$ is the bisimplical object
\[ \widetilde{X}_{n,m} = \coprod_ {i_0 \rightarrow \cdots \rightarrow i_n} X(i_0)_m.  \]

On point-wise fibrant\footnote{in case 4.\@ of \ref{PARMC} (Bousfield localization at the \v{C}ech covers) it suffices to have globally fibrant objects because of the exactness.} objects  $X \in \mathcal{M}^I$ the homotopy limit is given explicitly by
\begin{equation}\label{eqholim} \holim X = \int_{ I} \Hom( N(I \times_{/I} -),  X).  \end{equation}
\end{SATZ}
\begin{proof}
(Der1) and (Der2) are clear. 

(FDer0 left/right) is Corollary~\ref{KORBIFIB}. 

(FDer3 left/right) The formulas (\ref{eqhocolim}) and (\ref{eqholim}) are proven in \cite[Theorem~7.2]{Hor21}.
They are functorial in $I$. We can therefore define for a functor $\alpha: I \rightarrow J$, a diagram $S \in \mathcal{S}^J$, and an object $\mathcal{E} \in \DD(I)_{\alpha^*S}$ represented by a point-wise cofibrant object
\[ (\alpha_!^{(S)} \mathcal{E})(j)  := \hocolim_{I \times_{/J} j} f_{j,\bullet} \iota_j^* \mathcal{E} \]
where $\iota_j: I \times_{/J} j \rightarrow I$ is the canonical functor, $f_j: \iota_j^*\alpha^*S \rightarrow p_{I \times_{/J} j}^* S(j)$ is the canonical morphism, and by $\hocolim$ we understand the right hand side of (\ref{eqhocolim}). 
It is a formal verification as in \cite[Appendix A]{Hor21} that this functor is indeed left adjoint to $\alpha^*$ restricted to the fiber over $S$. We can reason analogously for the case of relative right homotopy Kan extensions.

(FDer4 left/right) is clear from the given explicit formula. 
\end{proof}
Note that in cases 3.\@ and 4.\@ of \ref{PARMC} all limits exist and thus the morphism of prederivators is a right fibered derivator on $\Dia$. 
In these we have a bifibration of model categories with all limits and colimits and one can thus alternatively cite \cite[Theorem~6.2]{Hor17b}. 

There are the following relations between the constructed fibered derivators:
\begin{SATZ}
Let $\mathcal{S}$ be a small category. 
\begin{enumerate}
\item If $\mathcal{S}$ has finite limits there is an equivalence of prederivators
\[ \DD_{\mathcal{S}^{\amalg, \Delta^{\op}}} \cong  \DD_{\mathcal{SET}^{\mathcal{S}^{\op} \times \Delta^{\op}}}. \]
In particular $\DD_{\mathcal{S}^{\amalg, \Delta^{\op}}}$ is a right fibered derivator with domain $\Dia$ even if not all limits exist in $\mathcal{S}$. 
\item There is an adjunction of derivators
\[ \xymatrix{( \DD_{\mathcal{SET}_{loc}^{\mathcal{S}^{\op} \times \Delta^{\op}}})_S \ar@<3pt>[rr]^R & &  \ar@<3pt>[ll]^L (\DD_{\mathcal{SET}^{\mathcal{S}^{\op} \times \Delta^{\op}}})_S  } \]
for every $S \in \mathcal{S}$ in which the right adjoint $R$ is fully-faithful and the left adjoint $L$ commutes with homotopically finite limits. 
\end{enumerate}
\end{SATZ}
\begin{proof}
Assertion 1.\@ is a consequence of Proposition~\ref{PROPSMALLS} while
2.\@ follows  from Theorem~\ref{SATZBLEXACT}.
\end{proof}

\begin{SATZ}
The fibered derivator $\DD_{\mathcal{SET}_{loc}^{\mathcal{S}^{\op} \times \Delta^{\op}}}$ is local w.r.t.\@ the chosen Grothendieck topology on $\mathcal{S}$ and, if $\mathcal{S}$ is extensive, 
$\DD_{\mathcal{S}^{\Delta^{\op}}}$ is local w.r.t.\@ the extensive topology on $\mathcal{S}$. 
\end{SATZ}

\begin{proof}This assertion follows from the three following lemmas. 
\end{proof}

\begin{LEMMA}
Let $\mathcal{M}$ be one of the model categories in \ref{PARMC} and consider 
 the fibered derivator $\DD_{\mathcal{M}} \rightarrow \SSS$.
\begin{enumerate}
\item
For a Cartesian square
\[ \xymatrix{ 
X \times_Y Z \ar[r]^G \ar[d]^F & X \ar[d]^f  \\
Z \ar[r]^g & Y }  \]
 in $\mathcal{S}$  the natural exchange morphism
\[  F_\bullet  G^\bullet \rightarrow g^\bullet f_\bullet  \]
is an isomorphism. 
\item
For any morphism $f$ in $\mathcal{S}$ the pull-back functor $f^\bullet$ commutes with homotopy colimits as well. 
\end{enumerate}
In other words, {\em all} morphisms in $\mathcal{S}$ are $\DD_{\mathcal{M}}$-local (cf.\@ Definition~\ref{DEFLOCAL}).
\end{LEMMA}
\begin{proof}
1.\@ The formula for the underived functors follows trivially from properties of fiber products. However all functors are equal to their derived functors because they map weak equivalences to weak equivalences. 
2.\@ follows from the explicit formula (\ref{eqhocolim}) of Theorem~\ref{SATZHOCOLIM} and the extensivity of $\mathcal{M}$.
\end{proof}

\begin{LEMMA}
Consider the fibered derivator $\DD_{\mathcal{SET}_{loc}^{\mathcal{S}^{\op} \times \Delta^{\op}}}$. 
If $f_i: U^{(i)} \rightarrow S$ form a cover in the Grothendieck pretopology then the functors $f_i^\bullet$ are jointly conservative.
\end{LEMMA}
\begin{proof}
Let $g: A_\bullet \rightarrow B_\bullet$ be a morphism in $\mathcal{SET}_{loc}^{\mathcal{S}^{\op} \times \Delta^{\op}}$ over $S$.
Form the coproduct $U:= \coprod_i U^{(i)}$ in $\mathcal{SET}^{\mathcal{S}^{\op}}$ and the \v{C}ech cover $f: U_\bullet \rightarrow S$.
Assume that $f_i^\bullet g$ is a weak equivalence for all $i$. Then $g \times_S U_k$ is a trivial cofibration for all $U_k$ (because the morphisms from the components of $U_k$ factor through one of the $f_i$). 

Form the morphism of bisimplicial objects
\[ \widetilde{A}_{i,j} = A_i \times_S U_j  \rightarrow  \widetilde{B}_{i,j} = B_i \times_S U_j  \] 
and consider $A$ and $B$ as bisimplicial presheaves constant in horizontal direction. We get a diagram
\[ \xymatrix{ 
\widetilde{A}_{\bullet,\bullet} \ar[r] \ar[d] & \widetilde{B}_{\bullet,\bullet} \ar[d]  \\
A_{\bullet,\bullet} \ar[r] & B_{\bullet,\bullet} }  \]
and the induced diagram
\[ \xymatrix{ 
\mathrm{hocolim} \widetilde{A}_{\bullet,\bullet} \ar[r] \ar[d] & \mathrm{hocolim} \widetilde{B}_{\bullet,\bullet} \ar[d]  \\
A_{\bullet} \ar[r]_g & B_{\bullet} }  \]
which is given by the diagonal simplicial presheaves\footnote{No cofibrancy assumption is needed because of the existence of the injective model category structure and its localization in which all objects are cofibrant (cf.\@ \cite[Lemma 3.16]{Hor21b}).}. 
The top horizontal morphism is a weak equivalences because the morphism $\widetilde{A}_{\bullet,\bullet} \rightarrow \widetilde{B}_{\bullet,\bullet}$ consists
of weak equivalences horizontally.
The vertical morphisms are weak equivalences because the columns themselves are pull-backs of the weak equivalence $U_\bullet \rightarrow S$ along $A_i \rightarrow S$, resp.\@ $B_i \rightarrow S$ and thus are weak equivalences in $\mathcal{SET}_{loc}^{\mathcal{S}^{\op} \times \Delta^{\op}}$ by the reasoning in the proof of Theorem~\ref{SATZBIFIBMC}. Therefore $g$ is a weak equivalence. 
\end{proof}

\begin{LEMMA}
Let $\mathcal{S}$ be an extensive category with finite limits and consider a family of coproduct injections $f_i: U^{(i)} \rightarrow \coprod_i U^{(i)}$. 
Then the functors $f_i^\bullet$ in the fibered derivator $\DD_{\mathcal{S}^{\Delta^{\op}}}$ are jointly conservative.
\end{LEMMA}
\begin{proof}
Let $g: A \rightarrow B$ be a morphism over $\coprod_i U^{(i)}$ in $\mathcal{S}^{\Delta^{\op}}$. We may assume w.l.o.g.\@ that $g$ is a cofibration.
Then by the extensivity it decomposes itself as a coproduct of 
morphisms $g_i: A^{(i)} \rightarrow B^{(i)}$ over $U^{(i)}$ in $\mathcal{S}^{\Delta^{\op}}$. The $g_i$ are itself also cofibrations. If
they become isomorphisms in $\DD_{\mathcal{S}^{\Delta^{\op}}}(\cdot)_{U^{(i)}}$ they are trivial cofibrations. Thus $g$, being a coproduct of trivial cofibrations, is 
a trivial cofibration itself and thus an isomorphism in $\DD_{\mathcal{S}^{\Delta^{\op}}}(\cdot)_{\coprod_i U^{(i)}}$.
\end{proof}

\begin{DEF}\label{DEFN}
For a diagram $(I, S)$ in $\Dia(\mathcal{S})$, consider the simplicial diagram $\Delta^{\op} \rightarrow \mathcal{S}^{\amalg}$ 
which, on the level of underlying simplicial sets, is given by the usual nerve $N(I)$ and whose entry in $\mathcal{S}$ at a
simplex $\alpha: [n] \rightarrow I$ is $S(\alpha(0))$. It is called the {\bf nerve} $N(I, S)$ of $(I,S)$. By abuse of notation we denote the same way its image in $\mathcal{S}^{\Delta^{\op}}$
under $\amalg$ and its image in $\mathcal{SET}^{\mathcal{S}^{\op} \times \Delta^{\op}}$ under $R$. 
\end{DEF}

Warning:  The Yoneda embedding $\mathcal{S}^{\Delta^{\op}} \rightarrow \mathcal{SET}^{\mathcal{S}^{\op} \times \Delta^{\op}}$ does not preserve $N(I, S)$!

\begin{PROP}\label{HOCOLIMNERVE}
Let $\mathcal{M}$ be one of the model categories in \ref{PARMC}.
Let $\DD_{\mathcal{M}} \rightarrow \SSS$ be the corresponding fibered derivator and let $S \in \mathcal{S}$.
Let $p: (I, F) \rightarrow (\cdot, S)$ be a diagram over $S$ and let $X \in \DD_{\mathcal{M}}(\cdot)_S$ be represented by a cofibrant object in $\mathcal{M}$. Then we have
\[ p_! p^* X = N(I, S) \times_S X.  \]
\end{PROP}
\begin{proof} 
$p_! p^*X$ is defined as $p_{I,!}f_\bullet f^\bullet p_I^*X$ where $f: F \rightarrow p_I^* S$ is the structural morphism.
The object
\[ f_\bullet f^\bullet p_I^* X  \]
is the diagram $S': I \rightarrow \mathcal{S}^{\amalg, \Delta^{\op}}$ which maps $i \in I$ to the simplicial diagram $F(i) \times_S X$ which is still cofibrant. 
Note that $f_\bullet$ and $f^\bullet$ are equal to their underived variants (cf.\@ Theorem~\ref{SATZBIFIBMC}). 
By the Bousfield-Kan formula  (\ref{eqhocolim})  of Theorem~\ref{SATZHOCOLIM} therefore
\[ p_{I,!} S' = N(I, S) \times_S X. \]
\end{proof}

\begin{KOR}\label{KOREQUIVALENCE}
Let $\mathcal{M}$ be one of the model categories in \ref{PARMC} and let  $S \in \mathcal{S}$. Then
 a morphism $\alpha: D_1 \rightarrow D_2$ in $\Dia(\mathcal{S})$ is a $\DD_{\mathcal{M}}$-equivalence over $S$ if and only if  $N(\alpha)$ is a weak equivalence in $\mathcal{M}$.
 In particular, the classes of morphisms $\alpha$ such that $N(\alpha)$ is a weak equivalence form a localizer in each of the cases w.r.t.\@ the appropriate Grothendieck topology. 
\end{KOR}

\section{The smallest localizer}

\begin{LEMMA}\label{LEMMANERVECOF}
Consider the categories $\mathcal{S}^{\amalg, \Delta^{\op}}$ and $\mathcal{S}^{\Delta^{\op}}$ with the split-projective structure. Then $N(I, S)$ is cofibrant in  $\mathcal{S}^{\amalg, \Delta^{\op}}$ and (its image under $\amalg$ is) cofibrant in $\mathcal{S}^{\Delta^{\op}}$
\end{LEMMA}
\begin{proof}
Compare the explicit description of cofibrant objects in Theorem~\ref{SATZEXTENSIVEMODELCAT} with the definition of $N(I,S)$. Furthermore $\amalg$ preserves cofibrant objects being left Quillen. 
\end{proof}

The following follows directly from the definition: 
\begin{LEMMA}\label{LEMMAPOINTWISENERVE}Recall from \ref{PARPOINTWISE} the definition of the $\Hom$-functor.
Let $X \in \mathcal{S}^{\amalg}$ and $(I, S) \in \Dia(\mathcal{S})$. Then
\[ \Hom_r(X, N(I, S)) = N(\Hom(X, (I, S))). \]
\end{LEMMA}

\begin{PROP}\label{PROPFWD}
We have a natural transformation $\int^\amalg N \Rightarrow \id$ whose entries
\[ \int^{\amalg} N(I, S) \rightarrow (I, S) \]
are of pure diagram type and are weak equivalences (i.e.\@ are contained in any localizer) for all $(I, S) \in \Dia(\mathcal{S})$. 
\end{PROP}
\begin{proof}
By (L4 left) is suffices to show that the diagram
\[ i \times_{/I} \int N(I)\]
is contractible. This is isomorphic to 
$\int N(i \times_{/I} I)$
which is contractible because $i \times_{/I} I$ has an initial object. 
\end{proof}

\begin{SATZ}\label{THEOREMBACK}
We have a natural transformation $N \int^\amalg \Rightarrow \id$ whose entries
\[  N\int^{\amalg} X_\bullet \rightarrow X_\bullet    \]
are weak equivalences for all $X_\bullet \in \mathcal{S}^{\amalg, \Delta^{\op}}$.
\end{SATZ}

\begin{proof}
The $k$-simplices $N(\int^{\amalg} X_\bullet)_k$ consist of the set of sequences
\[\xymatrix{  \Delta_{n_0} \ar[r]^{\alpha_1} & \Delta_{n_1} \ar[r]^{\alpha_2} &  \cdots  \ar[r]^{\alpha_k} &\Delta_{n_k}  }\]
together with an element $\xi$ in the underlying set of $X_{n_k}$ and the associated functor maps this to $X_{n_k}(\xi)$.
We define a morphism
\[ \Delta_k \rightarrow \Delta_{n_k} \]
mapping $i \in \Delta_k$ to $\alpha_{k}\cdots \alpha_{i+1}(n_i)$. Composing with the resulting $X_{n_k} \rightarrow X_k$ this defines a morphism 
\[ N\int^{\amalg} X_\bullet \rightarrow X_\bullet. \] Applying for any $S \in \mathcal{S}$
the functor $\Hom(S, -)$ we get, by Lemmas~\ref{LEMMAPOINTWISEINT} and \ref{LEMMAPOINTWISENERVE},
a morphism
\[ N \int \Hom(S, X_\bullet) \rightarrow \Hom(S, X_\bullet). \]
This is the same as the one considered in \cite[2.1.14]{Cis04} and it is thus a weak equivalence by a theorem of Quillen~\cite[2.1.15]{Cis04}.
\end{proof}

\begin{KOR}\label{KORINTWEGENERAL}Let $\mathcal{S}$ be a category with finite limits. 

The functor $\int^{\amalg}$ maps weak equivalences to weak equivalences in any localizer w.r.t.\@ the trivial topology in general (i.e.\@ not only between cofibrant objects).

Any localizer contains the class of morphisms $\alpha$ such that $N(\alpha)$ is a weak equivalence in $\mathcal{S}^{\amalg,\Delta^{\op}}$. This class is thus the smallest localizer w.r.t.\@ the trivial topology. 
The functors
\[ \xymatrix{ (\mathcal{S}^{\amalg, \Delta^{\op}}, \mathcal{W}) \ar@<3pt>[rr]^-{\int^{\amalg}}  &&  \ar@<3pt>[ll]^-{N} (\Dia(\mathcal{S}), \mathcal{W}_\infty) } \]
define an equivalence of categories with weak equivalences
where $\mathcal{W}_\infty$ is the smallest localizer w.r.t.\@ the trivial topology. 
\end{KOR}
\begin{proof}
We begin by showing that $\int^{\amalg}$ preserves weak equivalences (regardless of cofibrancy). Consider a weak equivalence $X_\bullet \rightarrow Y_\bullet$ and the diagram
\[ \xymatrix{ 
 N \int^{\amalg} X_\bullet \ar[r] \ar[d]&  X_\bullet \ar[d] \\
 N \int^{\amalg} Y_\bullet \ar[r] &  Y_\bullet
 }\]
 By Theorem~\ref{THEOREMBACK} the horizontal morphisms are weak equivalences. Hence by 2-out-of-3 the functor $N \circ \int^{\amalg}$ preserves weak equivalences. Consider the diagram
\[ \xymatrix{ 
\int^{\amalg} N \int^{\amalg} X_\bullet \ar[r] \ar[d]& \int^{\amalg} X_\bullet \ar[d] \\
\int^{\amalg} N \int^{\amalg} Y_\bullet \ar[r] & \int^{\amalg} Y_\bullet
 }\]
in which the horizontal morphisms are induced by the natural transformation $\int^{\amalg} N \rightarrow \id$, which has values in a weak equivalences by Proposition~\ref{PROPFWD}.
Since $N \int^{\amalg} X_\bullet \rightarrow N \int^{\amalg} Y_\bullet$ is a weak equivalence {\em between cofibrant objects} (Lemma~\ref{LEMMANERVECOF}), by Theorem~\ref{THEOREMINTWECOF}, the left vertical morphism is a weak equivalence. Hence by 2-out-of-3 also $\int^{\amalg} X_\bullet \rightarrow \int^{\amalg} Y_\bullet$ is a weak equivalence. 

We have to show that any localizer contains the class of $\DD_{\mathcal{S}^{\amalg,\Delta^{\op}}}$-equivalences which, by Corollary~\ref{KOREQUIVALENCE}, is the class of morphisms $\alpha$ such that $N(\alpha)$ is a weak equivalence in $\mathcal{S}^{\amalg,\Delta^{\op}}$. Since that class is itself a localizer w.r.t.\@ the trivial topology by Theorem~\ref{SATZWEAKHOMDESCENT}, it must be the smallest such localizer. 

Let $\alpha: D_1 \rightarrow D_2$ be a morphism such that $N(\alpha)$ is a weak equivalence and consider the diagram
\[ \xymatrix{ 
  \int^{\amalg} N D_1 \ar@{<-}[r] \ar[d]&  D_1 \ar[d] \\
 \int^{\amalg} N D_2  \ar@{<-}[r] &  D_2
 }\]
The left vertical morphism is a weak equivalence (i.e.\@ it is contained in the localizer) and the horizontal morphisms are weak equivalences by Proposition~\ref{PROPFWD}.
Thus also $D_1 \rightarrow D_2$ is in the localizer. 
\end{proof}

\begin{KOR}\label{KORINTWEGENERALCECH}Let $\mathcal{S}$ be an extensive category with finite limits and such that every object is a coproduct of $\N$-small objects. 

The functor $\int: \mathcal{S}^{\Delta^{\op}} \rightarrow \Dia(\mathcal{S})$ maps weak equivalences to weak equivalences in any localizer w.r.t.\@ the extensive topology.
Any localizer w.r.t.\@ the extensive topology contains the morphisms $\alpha$ such that $N(\alpha)$ is a weak equivalence in $\mathcal{S}^{\Delta^{\op}}$. This class is thus the smallest localizer w.r.t.\@ the extensive topology. 

The functors
\[ \xymatrix{ (\mathcal{S}^{\Delta^{\op}}, \mathcal{W}) \ar@<5pt>[rr]^-{\int}  &&  \ar@<5pt>[ll]^-{N} (\Dia(\mathcal{S}), \mathcal{W}_\infty) } \]
define an equivalence
of categories with weak equivalences
where $\mathcal{W}_\infty$ is the smallest localizer w.r.t.\@ the extensive topology. 
\end{KOR}

\begin{proof}
Consider a weak equivalence $X_\bullet \rightarrow Y_\bullet$ in $\mathcal{S}^{\Delta^{\op}}$ and factor it in  $\mathcal{S}^{\amalg, \Delta^{\op}}$ into a
trivial cofibration followed by a fibration:
\[ RX_\bullet \rightarrow X'_\bullet \rightarrow RY_\bullet \]
The functor $\amalg$ preserves cofibrations, trivial cofibrations, and fibrations \cite[Lemmas 6.3--6.4]{Hor21}. Hence applying $\amalg$ we get a factorization
\[ X_\bullet \rightarrow \amalg X'_\bullet \rightarrow Y_\bullet \]
into a trivial cofibration and a fibration and thus the fibration has to be trivial. 
Consider the diagram
\[ \xymatrix{ RX_\bullet \ar[r]^1 \ar@{=}[d] &  X'_\bullet \ar[r] \ar[d]_2 &  RY_\bullet  \ar@{=}[d] \\
 R \amalg RX_\bullet \ar[r] & R \amalg X'_\bullet \ar[r]_3 & R \amalg RY_\bullet }
\]
Morphism 1 is mapped by $\int^{\amalg}$ to a weak equivalence because it is a trivial cofibration. Morphism 2 is mapped by $\int^{\amalg}$ to a weak equivalence by Proposition~\ref{PROPPROPERTIESLOCALIZER}, 3.\@ because it is
point-wise in $\Delta^{\op}$ a weak equivalence (by (L3 left) involving the corresponding coproduct cover). Morphism 3 is a trivial fibration and thus mapped to a weak equivalence. Therefore $\int X_\bullet \rightarrow \int Y_\bullet$ is a weak equivalence. 

Furthermore, if $X_\bullet \in \mathcal{S}^{\Delta^{\op}}$, the morphism 
\[  R X_\bullet \rightarrow N( \int^{\amalg} R X_\bullet ) \]
is a weak equivalence in $\mathcal{S}^{\amalg, \Delta^{\op}}$. 
Therefore also 
\[  X_\bullet \rightarrow \amalg N( \int^{\amalg} R X_\bullet ) \]
is a weak equivalence in $\mathcal{S}^{\Delta^{\op}}$.

We have to show that any localizer contains the class of $\DD_{\mathcal{S}^{\Delta^{\op}}}$-equivalences which by Corollary~\ref{KOREQUIVALENCE} are the morphisms $\alpha$ such that $N(\alpha)$ is a weak equivalence in $\mathcal{S}^{\Delta^{\op}}$. Since it is by Theorem~\ref{SATZWEAKHOMDESCENT}  itself a localizer w.r.t.\@ the extensive topology it is the smallest such localizer. 
For the proof note that we have  a diagram (functorial in $(I, S)$)
\[ \int \amalg  N(I, S)\leftarrow \int^{\amalg} N(I, S) \rightarrow (I, S)  \]
in which the right morphism is a weak equivalence by Proposition~\ref{PROPFWD} and the left morphism is a
weak equivalence by the reasoning above. 
\end{proof}

\begin{PROP}\label{PROPCECH}
Let $X \in \mathcal{S}$ and let $U_\bullet \rightarrow X$ be a \v{C}ech covering in $\mathcal{S}^{\amalg, \Delta^{\op}}$. Let $U_\bullet \rightarrow  \widetilde{X} \rightarrow X$
be the factorization into a cofibration and a trivial fibration.
Then 
\[ (\partial \Delta_n \rightarrow  \Delta_n) \boxplus (U_\bullet \rightarrow \widetilde{X}) \]
is mapped  to a weak equivalence under $\int^\amalg$.
\end{PROP}
\begin{proof}
The morphism $U_\bullet \rightarrow X$ is mapped to a weak equivalence because after pull-back along $U_0 \rightarrow X$ it becomes a left homotopy equivalence and is thus a 
weak equivalence by Corollary~\ref{KORLEFTHOMOTOPYEQUIVALENCE} and (L3 left). 
By Proposition~\ref{PROPPROPERTIESLOCALIZERII}, 2.\@ also $K \otimes U_\bullet \rightarrow K \otimes X$ for any simplicial set $K$ is mapped to a weak equivalence. The morphism $K \otimes \widetilde{X} \rightarrow K \otimes X$ is a global weak equivalence and thus
it is mapped to a weak equivalence by Corollary~\ref{KORINTWEGENERAL}. Therefore also the cofibration $K \otimes U \rightarrow K \otimes \widetilde{X}$ is mapped to a weak equivalence.
 Finally the diagram  
\[ \xymatrix{
\partial \Delta_n \otimes U \ar[r]  \ar[d]_{\in \Cof} &  \Delta_n \otimes U  \ar[d]_{\in \Cof} \ar[rdd] \\
 \partial \Delta_n \otimes \widetilde{X} \ar[rrd] \ar[r] & \lefthalfcap  \ar[rd] \\
 & &  \Delta_n \otimes \widetilde{X}
} \]
shows that also $(\partial \Delta_n \rightarrow  \Delta_n) \boxplus (U \rightarrow \widetilde{X})$ is mapped to a weak equivalence. Note that the vertical morphisms are
cofibrations and thus the push-out exists and it is mapped to a weak equivalence by Proposition~\ref{PROPPUSHOUT}. 
\end{proof}

\begin{PAR}\label{PARINTCECH}
Let $\mathcal{S}$ be a small idempotent complete\footnote{For example $\mathcal{S}$ is idempotent complete if it has finite limits or finite colimits.  } category with Grothendieck pretopology and consider the left Bousfield localization 
$\mathcal{SET}_{loc}^{\mathcal{S}^{\op} \times \Delta^{\op}}$. Its cofibrant objects are the same as the cofibrant objects in $\mathcal{SET}^{\mathcal{S}^{\op} \times \Delta^{\op}}$ and thus these are in the essential image of $\mathcal{S}^{\amalg, \Delta^{\op}}$ because $\mathcal{S}$ is idempotent complete. Hence, on cofibrant objects, we have a functor
\[ \int^{\amalg}: \mathcal{SET}_{loc}^{\mathcal{S}^{\op} \times \Delta^{\op}, \Cof} \rightarrow \Dia(\mathcal{S}).  \]
\end{PAR}

\begin{SATZ}\label{THEOREMINTWECOFCECH}Let $\mathcal{S}$ be a small category with finite limits and Grothendieck pretopology.

The functor $\int^\amalg$  of (\ref{PARINTCECH}) maps (\v{C}ech) weak equivalences to weak equivalences in any localizer w.r.t.\@ the chosen Grothendieck topology. 
Any localizer w.r.t.\@ the given Grothendieck topology contains the morphisms $\alpha$ such that $N(\alpha)$ is a weak equivalence in $\mathcal{SET}_{loc}^{\mathcal{S}^{\op} \times \Delta^{\op}}$. This class is the smallest localizer w.r.t.\@ the given Grothendieck topology. 

The functors
\[ \xymatrix{ (\mathcal{SET}_{loc}^{\mathcal{S}^{\op} \times \Delta^{\op}}, \mathcal{W}) \ar@<5pt>[rr]^-{\int^{\amalg} Q}  &&  \ar@<5pt>[ll]^-{N} (\Dia(\mathcal{S}), \mathcal{W}_\infty) } \]
define an equivalence of categories with weak equivalence, 
where $Q$ denotes the cofibrant replacement functor,  $\int^{\amalg}$ is the functor described in \ref{PARINTCECH}, and $\mathcal{W}_\infty$ is the smallest localizer w.r.t.\@ the chosen Grothendieck topology. 
\end{SATZ}
\begin{proof}
(cf.\@ the proof of Theorem~\ref{THEOREMINTWECOF})
Let $f$ be a weak equivalence between cofibrant objects.
It can be factored as
\[ f = p \circ \iota \]
where $p$ is a trivial fibration between cofibrant objects and where $\iota$ is a transfinite composition of pushouts of morphisms of the form
\[ \Lambda_e( \Delta_n \times \Delta_1) \otimes S \rightarrow (\Delta_n \times \Delta_1) \otimes S  \] 
for $S \in \mathcal{S}$ and also 
\[ (\partial \Delta_{n} \hookrightarrow \Delta_n) \boxplus (U_\bullet \rightarrow X')  \] 
for a cofibration representing a \v{C}ech cover.
Everything thus takes place in the full subcategory $\mathcal{S}^{\amalg, \Delta^{\op}}$ (cf.\@ also Lemma~\ref{LEMMAFACTORCECH}) and push-outs along, and transfinite compositions of
cofibrations commute with the embedding $R$. 
The trivial fibration is mapped to a weak equivalence by Corollary~\ref{KORLEFTHOMOTOPYEQUIVALENCE}.
Therefore the first statement follows from Propositions~\ref{PROPPUSHOUT}--\ref{PROPTRANSFINITE} and \ref{PROPCECH}.

We have to show that any localizer contains the class of $\DD_{\mathcal{SET}_{loc}^{\mathcal{S}^{\op} \times \Delta^{\op}}}$-equivalences. By Corollary~\ref{KOREQUIVALENCE} those are the morphisms $\alpha$ such that $N(\alpha)$ is a weak equivalence in $\mathcal{SET}^{\mathcal{S}^{\op} \times \Delta^{\op}}_{loc}$. Since this class is  itself a localizer w.r.t.\@ the given Grothendieck topology by Theorem~\ref{SATZWEAKHOMDESCENT}, it is the smallest such localizer.
For the proof observe that we have the weak equivalence
\[ (I, S) \leftarrow  \int^{\amalg} N(I, S) \]
for any $(I, S) \in \Dia(\mathcal{S})$.
Hence, if $N(\alpha)$ is a \v{C}ech weak equivalence, $\alpha$ must be in the localizer. 

To prove that the functors $\int^{\amalg} Q$ and $N$ are mutually inverse equivalences of categories with weak equivalences observe that both functors preserve weak equivalences. 
Furthermore, we have natural transformations
\[ \id \leftarrow Q \rightarrow N \int^{\amalg} Q    \]
which are point-wise weak equivalences (already in $\mathcal{SET}^{\mathcal{S}^{\op} \times \Delta^{\op}}$, i.e.\@ globally).
Furthermore, we have the composition of natural transformations
\[ \id \leftarrow  \int^{\amalg} N  \leftarrow \int^{\amalg} Q N    \]
Note that $QN(I, S) \rightarrow N(I, S)$ is a global weak equivalence in the essential image of $\mathcal{S}^{\amalg, \Delta^{\op}}$ and thus mapped by $\int^{\amalg}$ to a weak equivalence. 
\end{proof}

\begin{BEM}
Unfortunately, if $\mathcal{S}$ is a big category with finite limits and {\em  non-trivial} Grothendieck topology, we are not able to describe the smallest localizer on $\Dia(\mathcal{S})$. 
In the best case there would be an explicit description of \v{C}ech weak equivalences in $\mathcal{S}^{\amalg, \Delta^{\op}}$ (not necessarily part of a model category structure), which does not depend on
the embedding into simplicial presheaves w.r.t.\@ a larger universe, and the smallest localizer on $\Dia(\mathcal{S})$ would be the class of morphisms $\alpha$ such that $N(\alpha)$ is
such a \v{C}ech weak equivalence. 
\end{BEM}

\section{The universal property revisited}

Whereas the use of $(\Dia(\mathcal{S}), \mathcal{W}_\infty)$, where $\mathcal{W}_\infty$ is the smallest localizer for the trivial topology, is very useful for studying fibered derivators, it follows from Theorem~\ref{THEOREMINTWECOFCECH} that  for small $\mathcal{S}$ it is equivalent to
to $(\Dia^{\mathcal{S}^{\op}}, \mathcal{W}_{\infty,\mathcal{S}^{\op}})$, where $\mathcal{W}_{\infty, \mathcal{S}^{\op}}$ is the class of morphisms that are point-wise in $\mathcal{S}^{\op}$ in the smallest localizer on $\Dia$. Cisinski showed that the associated derivator $\HH(\mathcal{S})$ of the latter satisfies the universal property
\[ \Hom_!(\HH(\mathcal{S}), \DD) \cong \DD_{\mathcal{S}} = \Hom(\SSS, \DD) \]
where $\Hom_!$ is the left derivator of continuous morphisms of left derivators\footnote{Cisinski uses a stronger definition (d\'erivateur a gauche) of left derivator which is however equivalent to the usual definition (d\'erivateur faible a gauche) by Theorem~\ref{FAIBLE}.}.

In this section we explain that this result follows also directly from the discussion in this article, {\em regardless of the size of $\mathcal{S}$}. We refer to Proposition~\ref{PROPDERLOCALIZER} in the appendix for the proof that the pair
$(\Dia(\mathcal{S}), \mathcal{W})$ yields a left derivator for arbitrary $\mathcal{S}$ and an arbitrary localizer $\mathcal{W}$, in which moreover the homotopy colimit is given by the Grothendieck construction~\ref{PARGCONSTR}.  For locally small $\mathcal{S}$ with finite limits and the smallest localizer $\mathcal{W}_\infty$ (for the trivial topology) this follows also from Theorem~\ref{KORINTWEGENERAL} and Theorem~\ref{SATZHOCOLIM}, giving the additional information that all values of the derivator are locally small.

\begin{SATZ}\label{SATZUNIV}
For any category $\mathcal{S}$ and for any left derivator $\DD$ we have an equivalence 
\[ \mathbb{HOM}_!(\HH(\mathcal{S}), \DD) \cong \mathbb{HOM}(\SSS, \DD) \]
where $\HH(\mathcal{S})$ is the left derivator associated with the pair $(\Dia(\mathcal{S}), \mathcal{W}_\infty)$ and $\SSS$ is the prederivator associated with $\mathcal{S}$.
The morphism from left to right is the restriction along the inclusion $\SSS \rightarrow \HH(\mathcal{S})$ and
the morphism from right to left (at the level of underlying categories) maps a morphism $\Xi: \SSS \rightarrow \DD$ of prederivators to 
\begin{equation}\label{eqmap}   (F, S) \mapsto p_{I,!} \Xi(S)  \end{equation}
where as usual an object in $\Dia(\mathcal{S})^I$ is given by a pair $(F, S)$ consisting of $F \in \Dia^I$ and $S \in \mathcal{S}^{\int F}$, and
where $p_I: \int F \rightarrow I$ denotes the defining opfibration. 
\end{SATZ}

\begin{proof}
It is clear that the two morphisms are inverse to each other up to 2-isomorphism. Hence
we only have to see that (\ref{eqmap}) is 1.\@ well-defined, i.e.\@ it maps weak equivalences in $\Dia(\mathcal{S})^I$ to isomorphisms, and 2.\@ it is continuous. 
We leave it to the reader to check that (\ref{eqmap}) is a morphism of prederivators. This holds because $p_{I,!}$ for an opfibration $p_I$ is ``computed point-wise''. 
The extension to diagrams of the equivalence is given by replacing $\DD$ with its fibers $\DD_J$.

1.\@ Since $p_I$ is an opfibration this immediately boils down to the statement for $I = \cdot$. Let $\mathcal{W}_\Xi$ be the class of morphisms $\alpha: (I, S) \rightarrow (J, T)$ in $\Dia(\mathcal{S})$ with the property that the induced morphism
\[ p_{I,!}\Xi(S) \rightarrow p_{J,!}\Xi(T)   \]
is an isomorphism. It suffices to show the following properties:
\begin{enumerate}
\item[(1)] The class $\mathcal{W}_\Xi$ is weakly saturated. 
\item[(2)] If $I \in \Dia$ has a final object $i$ and $S \in \mathcal{S}^I$ then the morphism $(I, S) \rightarrow (i, S(i))$ is in $\mathcal{W}_\Xi$.
\item[(3)] If $\alpha: D_1 \rightarrow D_2$ is a morphism over $D_3=(K, U)$ and if $\alpha \times_{/D_3} (k, U(k))$ is in $\mathcal{W}_\Xi$ for all $k \in K$ then $\alpha \in \mathcal{W}_\Xi$.
\item[(4)] If $\alpha: I \rightarrow J$ is a fibration with contractible fibers (in the sense that all $I_j \rightarrow \cdot$ are in the smallest localizer on $\Dia$) then $(I, \alpha^* T) \rightarrow (J, T)$ is in $\mathcal{W}_\Xi$ for all $T \in \mathcal{S}^J$. 
\end{enumerate} 
Note that axiom (4) is a bit different from the axiom (L4 left) of a localizer. Nevertheless it follows from the proofs that the smallest class satisfying (1--4) above coincides with the smallest localizer w.r.t.\@ the trivial topology. 

(1) and (2) are clear. 

(3) Let $\alpha: (I, S) \rightarrow (J, T)$ be the morphism, then $\alpha \times_{/D_3} (k, U(k))$ is the morphism
\[ (I \times_{/K} k, \iota_I^*S) \rightarrow (J \times_{/K} k, \iota_J^*S) \]
where $\iota_I: I \times_{/K} k \rightarrow I$ is the projection and  similarly for $\iota_J$.
It suffices to show that 
\[ p_{1,I} \Xi(S) \rightarrow  p_{2,I} \Xi(T) \]
is an isomorphism. This can be shown point-wise by (Der2) but 
\[ k^* p_{1,I} \Xi(S) \rightarrow  k^* p_{2,I} \Xi(T) \]
is the same as
\[ p_{I \times_{/K} k,!} \Xi(\iota_I^*S) \rightarrow   p_{J \times_{/K} k,!} \Xi(\iota_J^*T) \]
which is an isomorphism by assumption. 

(4) follows immediately from Theorem~\ref{FAIBLE}.

2. By Theorem~\ref{PROPDERLOCALIZER} the homotopy colimit in $\Dia(\mathcal{S})$ w.r.t.\@ the smallest localizer (in fact any localizer) is given by the Grothendieck construction. By definition, the image of $F: I \rightarrow \Dia(\mathcal{S})$ in $\DD(I)$  is given by
\[ p_{I,!} \Xi(S), \]
where $p_I: \int F \rightarrow I$ is the opfibration and where $S \in \mathcal{S^{\int F}}$. 
Its homotopy colimit is obviously the same as the image of $\int F$ in $\DD(\cdot)$.
\end{proof}

\begin{PAR}
To compare the morphism (\ref{eqmap}) with Cisinski's construction {\em for small $\mathcal{S}$} observe that we have an equivalence
\[ \Hom: (\Dia(\mathcal{S}), \mathcal{W}_\infty) \rightarrow (\Dia^{\mathcal{S}^{\op}},\mathcal{W}_{\infty, \mathcal{S}^{\op}})    \]
in which $\mathcal{W}_\infty$ on the left is the smallest localizer on $\Dia(\mathcal{S})$ and $\mathcal{W}_{\infty, \mathcal{S}^{\op}}$ is the class of morphisms that are point-wise in $\mathcal{S}^{\op}$ in
the smallest localizer on $\Dia$. The functor is the functor $\Hom: \mathcal{S}^{\op} \times \Dia(\mathcal{S}) \rightarrow \Dia$ defined in \ref{PARPOINTWISE}.
Cisinski defines the morphism (on the underlying categories) by 
\begin{eqnarray*} \mathrm{Cis}: \Dia^{\mathcal{S}^{\op}} \times \DD(\mathcal{S}) & \rightarrow & \DD(\cdot) \\
 F, \Xi & \mapsto & p_{\nabla F,!} p_{\mathcal{S}}^*  \Xi 
\end{eqnarray*}
where $p_{\mathcal{S}}: \nabla F \rightarrow \mathcal{S}$ is the canonical fibration. 

Let us check that the diagram
\begin{equation} \vcenter{ \xymatrix{
\Dia(\mathcal{S}) \times \DD(\mathcal{S}) \ar[d]_-{\Hom \times \id} \ar[rrd]^-{\qquad ((I, S), \Xi) \mapsto p_{I,!}S^*\Xi} & \\
\Dia^{\mathcal{S}^{\op}} \times \DD(\mathcal{S})  \ar[rr]_-{\mathrm{Cis}} && \DD(\cdot)
} } \end{equation} 
commutes up to a canonical isomorphism. For $F=(I,S) \in \Dia(\mathcal{S})$ the fibration
\[ \nabla_{\mathcal{S}} \Hom(-, F) \rightarrow \mathcal{S}  \]
is the same as the fibration defined by the comma category:
\[ \pr_1: \mathcal{S} \times_{/\mathcal{S}} I \rightarrow \mathcal{S} \]
and we have
\[ \pr_{\mathcal{S} \times_{/\mathcal{S}} I,!} \pr_1^* = p_{I,!} \pr_{2,!} \pr_1^* \cong p_{I,!} S^*  \]
because the diagram
\[ \xymatrix{\mathcal{S} \times_{/\mathcal{S}} I \ar[d]_{\pr_2} \ar[r]^-{\pr_1} & \mathcal{S} \ar@{=}[d] \\
I \ar[r]_S & \mathcal{S}
 } \]
 is homotopy exact. Therefore the diagram commutes up to canonical isomorphism. 
\end{PAR}

\section{Application: Extension of fibered derivators}\label{SECTEX}

Let $\mathcal{S}$ be a small category with finite limits and Grothendieck pretopology. Let $\SSS$ be the prederivator represented by $\mathcal{S}$. 
The goal of this section is to prove the following theorem:

\begin{SATZ}\label{SATZEXDER}
Let $\DD \rightarrow \SSS$ be an infinite fibered derivator which is local w.r.t.\@ the Grothendieck
pretopology on $\mathcal{S}$ with stable, well-generated fibers. Then there is a natural extension 
\[ \DD'' \rightarrow \mathbb{H}^2(\mathcal{SET}^{\Delta^{\op} \times \mathcal{S}^{\op}}_{loc}) \]
such that the pullback of $\DD''$ along the natural morphism 
\[ \SSS \rightarrow \mathbb{H}^2(\mathcal{SET}^{\Delta^{\op} \times \mathcal{S}^{\op}}_{loc})  \]
is equivalent to $\DD$ and such that for a pair
$(I, S) \in \Dia(\mathcal{S})$ with $I \in \Dia$ and $S \in \mathcal{S}^{I}$
we have an equivalence of stable derivators
\[ \DD''_{N(I, S)} \cong  \DD_{(I, S)}^{\cart}. \]
\end{SATZ}

Recall (compare also with Definition~\ref{DEFWEAKDEQUIV}):

\begin{DEF}\label{DEFSTRONGDEQUIV}
Let $\DD \rightarrow \SSS$ be a left (resp.\@ right) fibered derivator satisfying also (FDer0 right) (resp.\@ (FDer0 left)). A morphism 
\[ \xymatrix{ 
D_1 \ar[rr]^\alpha & & D_2 
 } \]
in $\Dia(\mathcal{S})$ (resp.\@ in $\Dia^{\op}(\mathcal{S})$) is called a {\bf strong $\DD$-equivalence} over $S$ if $\alpha^*$ induces an equivalence
\[ \xymatrix{ \DD(D_2)^{\cart} \ar[r]^{\alpha^*} & \DD(D_1)^{\cart} } \]
resp.\@
\[ \xymatrix{ \DD(D_2)^{\cocart} \ar[r]^{\alpha^*} &  \DD(D_1)^{\cocart} } \]
\end{DEF}
A strong $\DD$-equivalence is, in both cases, also a weak $\DD$-equivalence in the sense of Definition~\ref{DEFWEAKDEQUIV} over any $S$ but the converse does not necessarily hold.
 
We recall the Main Theorem of homological descent \cite[Main Theorem 3.5.5]{Hor15}:
\begin{SATZ}\label{HAUPTSATZHOMDESC}
Let $\DD \rightarrow \SSS$ be an infinite fibered derivator which is local w.r.t.\@ the Grothendieck
pretopology on $\mathcal{S}$ with stable, well-generated fibers. Then
the strong $\DD$-equivalences form a localizer in the sense of Definition~\ref{DEFFUNLOC}.
\end{SATZ}
Well-generated fibers are needed because in the proof a Brown representability theorem is used. 
There is the following dual variant \cite[Main Theorem 3.5.4]{Hor15} (Main Theorem of cohomological descent):
\begin{SATZ}
Let $\DD \rightarrow \SSS^{\op}$ be an infinite fibered derivator which is colocal w.r.t.\@ the Grothendieck
pretopology on $\mathcal{S}$ with stable, compactly generated fibers. Then
the strong $\DD$-equivalences in $\Dia^{\op}(\mathcal{S}^{\op})$ form a localizer in the sense of Definition~\ref{DEFFUNLOC} (transported under the isomorphism ``opposite'' $\Dia^{\op}(\mathcal{S}^{\op}) \rightarrow \Dia(\mathcal{S})^{2-\op}$).
\end{SATZ}
Compactly generated fibers are needed because in the proof a Brown representability theorem for the dual is used. 

Recall from \cite[Definition 3.1]{Hor15b} the definition of the 2-category $\mathcal{M}^{cor}$ for a category $\mathcal{M}$ with fiber products.
We need the following variant: 

\begin{DEF}\label{DEFSCORFIB}Let $(\mathcal{M}, \Fib, \mathcal{W})$ be a {\em category of fibrant objects}. 
We define the following 2-category $\mathcal{M}^{cor, fib}$: 
\begin{enumerate}
\item The objects are the objects of $\mathcal{M}$.
\item The 1-morphisms in $\Hom(S, T)$ are the correspondences
\[ \xymatrix{ &U  \ar[dl]_{\alpha}  \ar[dr]^{\beta} \\ 
S && T.
} \]
in which $\beta$ is arbitrary and $\alpha$ is a trivial fibration. 

\item The 2-morphisms $(U, \alpha, \beta) \Rightarrow (U', \alpha', \beta')$ are the trivial fibrations $\gamma: U \rightarrow U'$ such that in
\begin{equation}\label{scor2morph} \vcenter{ \xymatrix{ 
&U  \ar[dl]_{\alpha} \ar[dd]^\gamma \ar[dr]^{\beta} \\ 
S  && T \\
&U' \ar[ul]^{\alpha'}  \ar[ur]_{\beta'}
} } \end{equation}
both triangles commute. 
\item The composition of 1-morphisms is given by the fiber product. 
\end{enumerate}
\end{DEF}

Literally this defines only a bicategory which we assume has been strictified. See \cite[5.3]{Hor17} for a similar strictification more generally for the
2-multicategory $\mathcal{M}^{cor}$. We refer to \cite[Section 2]{Hor16} for the discussion of pre-2-derivators and fibered derivators over them. 

\begin{DEF}\label{DEFH2}
Let $\mathcal{M}$ be a category of fibrant objects. We define the pre-2-derivator:
\[ I \mapsto \mathbb{H}^2(\mathcal{M})(I) \]
where $\mathbb{H}^2(\mathcal{M})(I)$ is the 2-category of pseudo-functors $I \rightarrow \mathcal{M}^{cor, fib}$, natural transfomations, and modifications, in which the
morphism categories $\Hom(F, G)$ for two pseudo-functors $F: I \rightarrow \mathcal{M}^{cor, fib}$ and $G: I \rightarrow \mathcal{M}^{cor, fib}$ are turned
into groupoids inverting formally all modifications.  
\end{DEF}

\begin{PAR}\label{PARSHARP}
Recall that a morphism $f$ in a model category is a {\bf sharp} morphism if any pull-back of $f$ has the property of preserving weak equivalences under pull-back. 
An object $X$ is called sharp-fibrant if the morphism $f: X \rightarrow \cdot$ to the final object is sharp.
If in $\mathcal{M}$ finite products are homotopy products (i.e.\@ if the product of two weak equivalences is a weak equivalence) then
all objects are sharp-fibrant. For example $\mathcal{SET}_{loc}^{\mathcal{S}^{\op} \times \Delta^{\op}}$ is right proper and every object is sharp-fibrant.
\end{PAR}
\begin{PROP}
In a model category $\mathcal{M}$ the full subcategory of fibrant objects forms a category with fibrant objects. 
In a right proper model category $\mathcal{M}$ the full subcategory of sharp-fibrant objects forms a category with fibrant objects. 
\end{PROP}

\begin{PAR} \label{PAREXT1}
Let $\DD \rightarrow \SSS$ be a fibered derivator. 
We first define an extension 
$\DD' \rightarrow \mathbb{CAT}(\mathcal{S})$ where $\mathbb{CAT}(\mathcal{S})$ is the usual pre-1-derivator associated with $\Dia(\mathcal{S})$ considered as 1-category, 
as follows: For a diagram $I$ we define 
\[ \DD'(I) :=  \{(F, X) \ | \ F: I \rightarrow \Dia, X \in \DD(\int F)^{\pi_I-\cart} \}  \]
where $\pi_I: \int F \rightarrow I$ is the canonical opfibration. The structural functor to $\Dia(\mathcal{S})^I$ is given via the identification
\begin{equation}\label{eqiddiasi} \Dia(\mathcal{S})^I = \{(F, X) \ | \ F: I \rightarrow \Dia, X \in \mathcal{S}^{\int F} \}.  \end{equation}
A functor $\alpha: I \rightarrow J$ induces a corresponding functor $\alpha': \int \alpha \circ F \rightarrow \int F$ and a {\em Cartesian} diagram
\[  \xymatrix{ \int \alpha \circ F \ar[d]_{\alpha'} \ar[r]^-{\pi_I} & I \ar[d]^{\alpha} \\
\int F \ar[r]_-{\pi_J} & J } \]
We define 
\[ \alpha^*(F, X) := (F \circ \alpha, (\alpha')^*X). \] 
The functor $(\alpha')^*$ clearly maps $\pi_J$-Cartesian objects to $\pi_I$-Cartesian objects. 
For a natural transformation $\mu: \alpha \Rightarrow \beta$, 
we get morphisms $F \circ \alpha \rightarrow F \circ \beta$ and $\DD(\mu')(X): (\alpha')^*X \rightarrow (\beta')^*X$ which provides the 2-functoriality of $\DD'$. 
\end{PAR}

\begin{PROP}\label{PROPEXDIA}
Let $\DD \rightarrow \SSS$ be a fibered derivator such that
left Cartesian projectors exist\footnote{This means that the inclusion functors $\DD(\int F)^{\pi_I-\cart}_{S} \rightarrow \DD(\int F)_{S}$ have left adjoints $\Box_!$ for every $F: I \rightarrow \Dia$ and $S \in \mathcal{S}^{\int F}$. } (for instance, if $\DD \rightarrow \SSS$ satisfies the assumptions of Theorem~\ref{HAUPTSATZHOMDESC}). The morphism 
of prederivators defined in \ref{PAREXT1}
\[ \DD' \rightarrow \mathbb{CAT}(\mathcal{S}) \]
is again a fibered derivator.  Furthermore, for each $I \in \Dia$ and morphism $f$ in $\Dia(\mathcal{S})^I$ which is point-wise a strong $\DD$-equivalences the pull-back
$f^\bullet$ is an equivalence.  
\end{PROP}
\begin{proof}We check the axioms of a fibered derivator. 

(Der1) and (Der2) for $\DD'$ follow immediately from the corresponding axiom for $\DD$. 

(FDer0 left) The fiber over a functor $F: I \rightarrow \Dia(\mathcal{S})$ is given by 
\[ \DD\left(\int F'\right)_{S}^{\pi_I-\cart} \]
identifying $F$ with a pair $(F', S)$ as in (\ref{eqiddiasi}).

Let $\xi: F \rightarrow G$ be a morphism in $\Dia(\mathcal{S})^I$. By (\ref{eqiddiasi}) 
we can identify $F$ with a pair $(F', S)$ with $F' \in \Dia^I$ and $S: \int F' \rightarrow \mathcal{S}$ and $G$ with a pair
$(G', T)$ with $G' \in \Dia^I$ and $T: \int G' \rightarrow \mathcal{S}$.
The morphism $\xi$ is then given by a pair $(\mu, f)$ where $\mu: F' \rightarrow G'$ is a (strict) natural transformation
 and $f: S \rightarrow (\int \mu)^*T$ is a morphism. By construction the functor 
\[ \DD'(I) \rightarrow \Dia(\mathcal{S})^I \]
is a fibration with pull-back functors $\xi^\bullet$ given by the functoriality of the Grothendieck construction, namely by the composition: 
\[ (\int \xi)^*: \xymatrix{  \DD(\int G')_{T}^{\pi_I-\cart}  \ar[r]^-{(\int \mu)^*}  &  \ar[r]^-{f^\bullet}  \DD(\int F')_{(\int \mu)^*T}^{\pi_I-\cart} &  \DD(\int F')_{S}^{\pi_I-\cart}.   } \]
This functor obviously commutes with $(\alpha')^*$ and thus (FDer0 left) holds. 

(FDer0 right)
Neglecting the Cartesianity condition the functor $\xi^\bullet=(\int \xi)^*$ has the left adjoint $(\int \mu)_!^{(T)} f_\bullet$ which does not preserve the Cartesianity condition in general. 
Because of the existence of left Cartesian projectors $\Box_!$ we have adjoints on the subcategories of Cartesian objects given by
\[ \Box_! (\int \mu)_!^{(T)} f_\bullet \]
We have to show that these commute with $\alpha^*$. By (Der2) it suffices to see this for $i^*$ where $i: \cdot \rightarrow I$ is an object. 
Note that $i^*$ has been defined as the restriction of
\[ \xymatrix{  \DD(\int F')_{(\int F', S)}  \ar[r]^-{(i')^*} &  \DD(F(i))_{(F'(i), S|_{F'(i)})} }  \]
to Cartesian objects, where $i': F(i) \rightarrow \int F'$ is the inclusion of the fiber of $\int F' \rightarrow I$ above $i$. 
However, $(\int \mu)_!^{(T)}$ commutes with $(i')^*$ because $\int \mu$ is an opfibration, and $f_\bullet$ commutes with $(i')^*$ because of (FDer0 left) for $\DD \rightarrow \SSS$. It therefore suffices to see that functor $(i')^*$ commutes with the left Cartesian projector $\Box_!$. This is \cite[Lemma~3.5.9]{Hor15}.
Therefore also (FDer0 right) holds. 

(FDer3 left) follows as before using the left Cartesian projectors. 

(FDer3 right) Actually, $(\alpha')_*^{(G',T)}$ preserves Cartesianity. For note that this is equivalent to the commutation of the left Cartesian projector with $(\alpha')^*$, i.e.\@ the fact that the natural exchange morphism 
\begin{equation} \label{mor1}
\Box_!  (\alpha')^* \rightarrow  (\alpha')^*\Box_!  
\end{equation}
is an isomorphism. For each object $i \in I$ apply $(i')^*$:
\[ (i')^* \Box_!  (\alpha')^* \rightarrow (i')^* (\alpha')^*\Box_!.  \]
Since $(i')^*$ commutes with $\Box_!$ by \cite[Lemma~3.5.9]{Hor15} ($i'$ is the injection of a fiber into a Grothendieck opfibration) this is obviously an isomorphism. 
By (Der2) the morphism (\ref{mor1}) is thus an isomorphism. (Actually the Key Lemma is proven by establishing that $(i')_*$ preserves Cartesianity and the proof can probably extended to arbitrary $\alpha_*$ without reference to (Der2)).

Furthermore to establish (FDer4 left) and (FDer4 right) it suffices to show that for a Grothendieck fibration $\alpha: I \rightarrow J$ and $(G', T) \in \Dia(\mathcal{S})^J$ the canonical exchange morphism
\[     i^* \alpha_*^{(G',T)} \rightarrow  p_{I_j,*}^{(G'(i),T|_{G'(i)})} \iota^* \]
is an isomorphism. 
By definition this means that for the original fibered derivator $\DD \rightarrow \SSS$, the morphism 
\[  (i')^* (\alpha')_*^{(T)}  \rightarrow \pr_{2,*}^{(T|_{G'(i)})} (\iota')^*   \]
where $\pr_2: I_j \times G'(j) \rightarrow G'(j)$ is the second projection, has to be an isomorphism. 
However  the diagram
\[\xymatrix{  I_j \times G'(i) \ar[r] \ar[d]_{\pr_2} &  \int G' \circ \alpha \ar[d]^{\alpha'} \\
G'(i) \ar[r] & \int G'  } \]
is Cartesian and $\alpha'$ is again a fibration because the whole diagram is the pull-back of  
\[\xymatrix{  I_j \ar[r] \ar[d]&  I \ar[d]^{\alpha} \\
j \ar[r] & J  } \]
along $\pi_J$.  Therefore the requested property follows from (FDer3 right) for $\DD \rightarrow \SSS$. 

For the last assertion observe that if $f: F \rightarrow G$ is a morphism in $\Dia(\mathcal{S})^I$ which is a $\DD$-equivalence point-wise in $I$ then also $\int f: \int F \rightarrow \int G$ induces
an equivalence
\[ (\int f)^*: \DD(\int F)^{\pi_I-\cart} \rightarrow \DD(\int G)^{\pi_I-\cart}   \] 
by the proof of the Main Theorem of homological descent \cite[p.1321]{Hor15}. However, $f^\bullet$ in $\DD'$ is by the reasoning in the above proof of (FDer0 left) given by $(\int f)^*$. 
\end{proof}

\begin{PROP}\label{PROPKERNELPROJECTION}
Let
\[  \xymatrix{  \mathcal{D} \ar@/^5pt/[r]^F &  \mathcal{E} \ar@/^5pt/[l]^G } \]
 be an adjunction between well-generated triangulated categories with small coproducts in which
$F$ and $G$ are exact functors, and with $F$ right adjoint (resp.\@ left adjoint). 
Then there is exists a left adjoint (resp.\@ right adjoint) to the inclusion 
\[ \ker(F) \hookrightarrow \mathcal{D}. \]
Furthermore the category $\ker(F)$ is triangulated, and well-generated (or compactly generated, if $\mathcal{E}$ and $\mathcal{D}$ are). 
\end{PROP}
\begin{proof}
Without the additional statement this is \cite[Proposition 4.3.1]{Hor15}. If $F$ is the right adjoint then, by \cite[Proposition 7.2.1]{Kra09} not only the category $\langle G(\mathcal{E}_0) \rangle$ constructed in the proof of the proposition is well-generated but also the quotient  $\mathcal{D} / \langle G(\mathcal{E}_0) \rangle$. By \cite[Proposition 4.9.1]{Kra09} this is equivalent to $\ker F \cong \langle G(\mathcal{E}_0) \rangle^\perp$. 
The category $\ker F$ is thus well-generated. 
If $F$ is the left adjoint, then this follows from  \cite[Proposition 7.4.1]{Kra09} as stated  already in the proof of \cite[Proposition 4.3.1]{Hor15}. The statement involving compactly generated follow by setting the cardial $\alpha$ involved in the statements about $\alpha$-well-generated to $\aleph_0$. 
\end{proof}

\begin{LEMMA}\label{LEMMARIGHTPROJECTOR1}
Let $I \in \Dia$, let $\DD \rightarrow \SSS$ be a fibered derivator with domain $\Dia$ with stable, well-generated fibers, $S \in \SSS(I)$, and $M_{\cart} \subset \Mor(I)$ and $M_{\cocart} \subset \Mor(I)$ two subsets of morphisms. 
For each morphism $m: x \rightarrow y$ in $M_{\cart}$ consider the morphism $f_m: m^*S \rightarrow p_{\to}^*y^*S$ and
for each morphism $m: x \rightarrow y$ in $M_{\cocart}$ consider the morphism $g_m: p_{\to}^*x^*S \rightarrow m^*S$, both in $\SSS(\rightarrow)$. Assume that all 
$f_m^\bullet$ have right adjoints (resp.\@ that all $g_{m,\bullet}$ have left adjoints). Then 
 the subcategory 
\[ \DD(I)_S^{M_{\cart}, M_{\cocart}} = \{ \mathcal{F} \in \DD(I)_S \ |\ x^*\mathcal{F} \rightarrow y^*\mathcal{F} \text{ is } \substack{ \text{ Cartesian for all $m \in M_{\cart}$ } \\ \text{ coCartesian for all $m \in M_{\cocart}$ }  }    \} \]
is triangulated, well-generated (or compactly generated, if $\DD \rightarrow \SSS$ has compactly generated fibers), and the inclusion
\[ \DD^{M_{\cart}, M_{\cocart}}(I)_S \hookrightarrow \DD(I)_S \]
 has a right adjoint, resp.\@ a left adjoint.  
\end{LEMMA}
\begin{proof}
Cf.\@ \cite[Theorem~4.3.4]{Hor15}. We construct a functor $F$, composition of
\[\xymatrix{ \DD(I)_S \ar[d]^-{\prod_{m} m^*} \\ \prod_{m \in M_{\cart}}  \DD(\rightarrow)_{m^*S} \times \prod_{m \in M_{\cocart}}  \DD(\rightarrow)_{m^*S}   \ar[d]^-{\prod f_m^\bullet \times \prod g_{m,\bullet}} \\   \prod_{m \in M_{\cart}}  \DD(\rightarrow)_{p_{\to}^*x^*S} \times \prod_{m \in M_{\cocart}}  \DD(\rightarrow)_{p_{\to}^*y^*S} \ar[d]^-{\prod C} \\  \prod_{m \in M_{\cart}}  \DD(\cdot)_{x^*S} \times \prod_{m \in M_{\cocart}}  \DD(\cdot)_{y^*S}  }   \]
where the $C: \DD(\rightarrow)_{\cdots} \rightarrow \DD(\cdot)_{\cdots}$ is the  cone in the respective fiber. 
Note that $\ker(C)$ is the subcategory of $\DD(\rightarrow)_{\cdots}$ of objects whose underlying diagram consists of an isomorphism. 
Thus by construction 
\[ \DD^{M_{\cart}, M_{\cocart}}(I)  = \ker(F). \]
The statement follows from Proposition~\ref{PROPKERNELPROJECTION} because, by assumption, all functors that $F$ is composed of are left adjoints (resp.\@ right adjoints). 
\end{proof}

\begin{PROP}\label{PROPEXH2}
Let $(\mathcal{M}, \Fib,  \mathcal{W})$ be a category with fibrant objects and let $\mathbb{M}$ the prederivator represented by $\mathcal{M}$. 
Let $\DD' \rightarrow \mathbb{M}$ be a left and right fibered derivator such that for all morphisms $f: X \rightarrow Y$ in $\mathcal{M}^I$ point-wise in $\mathcal{W}$ the pull-back
$f^\bullet$ is an equivalence $\DD'(I)_Y \cong \DD'(I)_X$. Then there is an extension (well-defined up to equivalence) of $\DD'$ to a  left fibered derivator (satisfying also FDer0 right)
\[ \DD'' \rightarrow \mathbb{H}^2(\mathcal{M}) \]
such that its pull-back under the embedding $\mathbb{M} \rightarrow \mathbb{H}^2(\mathcal{M})$ is equivalent to $\DD'$. If $\DD'$ is infinite, and has stable and perfectly generated fibers then $\DD''$ is right fibered as well. 
\end{PROP}
\begin{proof}
This follows from \cite[Main Theorem 7.2]{Hor16}. Note that the pseudo-functor  
\[ \mathcal{M}^{cor,fib}(I) \rightarrow \mathcal{CAT} \]
constructed in the proof of \cite[Main Theorem 7.2]{Hor16} maps the 2-morphisms to natural isomorphisms because the corresponding units are isomorphisms by assumption, hence it induces a pseudo-functor
\[ \mathbb{H}^2(\mathcal{M})(I) \rightarrow \mathcal{CAT}. \]
\end{proof}

\begin{proof}[Proof of Theorem~\ref{SATZEXDER}.]
Consider the model category $\mathcal{SET}_{loc}^{\mathcal{S}^{\op}\times \Delta^{\op}}$. It can be turned into a category with fibrant objects
using all objects and the sharp morphisms as fibrations (cf.\@ \ref{PARSHARP}). One can also restrict to fibrant objects in $\mathcal{SET}^{\mathcal{S}^{\op}\times \Delta^{\op}}$
and consider fibrations in $\mathcal{SET}^{\mathcal{S}^{\op}\times \Delta^{\op}}$ or restrict to fibrant objects in $\mathcal{SET}_{loc}^{\mathcal{S}^{\op}\times \Delta^{\op}}$ and consider fibrations in $\mathcal{SET}_{loc}^{\mathcal{S}^{\op}\times \Delta^{\op}}$.
Denote by $\mathcal{M}$ any of these categories with fibrant objects.

{\bf Step 1:} Extend $\DD \rightarrow \SSS$ to a fibered derivator 
\[ \DD' \rightarrow \mathbb{CAT}(\mathcal{S}) \]
using Theorem~\ref{PROPEXDIA}. 
The fibered derivator $\DD'$ has again well-generated fibers as follows from Lemma~\ref{LEMMARIGHTPROJECTOR1}.

{\bf Step 2:} We pull-back $\DD'$ to $\mathbb{M}$ along the functor $\int^{\amalg} Q: \mathbb{M} \rightarrow \mathbb{CAT}(\mathcal{S})$ of \ref{PARINTCECH} (where $Q$ denotes the cofibrant replacement) and
extend it using Proposition~\ref{PROPEXH2} to $\mathbb{H}^{2}(\mathcal{M})$:
\[  \DD'' \rightarrow \mathbb{H}^2(\mathcal{M}).  \]
This is again a left and right fibered derivator because $\DD'$ is infinite with stable, well-generated (and hence perfectly generated) fibers. 

There are morphisms of pre-2-derivators
\begin{equation}\label{eqfunct7}  \mathbb{CAT}(\mathcal{S}) \rightarrow \mathbb{M} \rightarrow \mathbb{H}^{2}(\mathcal{M}) \end{equation}
where the first morphism is the nerve $N$ (possibly followed by a fibrant replacement, depending on the variant used) and the second morphism is the natural inclusion. 
By Theorem~\ref{THEOREMINTWECOFCECH} and its proof, for $D \in \Dia(\mathcal{S})^I$, there is a morphism $f_D: \int^{\amalg} Q N(D) \rightarrow D$ (the functor $\int^{\amalg} Q N$ is applied point-wise in $I$) which is (point-wise in $I$) in any localizer. Therefore also $\int f_D$ is in the localizer of strong $\DD$-equivalences and induces an equivalence:
\[ f_D^*: \DD'(I)_D = \DD(\int D)^{\pi_I-\cart} \cong \DD(\int \int^{\amalg} Q N(D))^{\pi_I-\cart} = \DD''(I)_{N(D)}.  \]
The pull-back of $\DD''$ along (\ref{eqfunct7}) is thus equivalent to the extension $\DD'$ of $\DD$. 
In particular, the pullback to $\SSS \subset \mathbb{CAT}(\mathcal{S})$ is equivalent to $\DD$.
\end{proof}

\appendix

\section{Saturatedness of localizers and the associated left derivator}

We give an elementary proof of the following theorem, which has been proven for $\mathcal{S}= \{\cdot\}$ (i.e.\@ in the case of basic localizers of Grothendieck) by Cisinski \cite[Proposition~4.2.4]{Cis06}.

\begin{SATZ}\label{SATZSAT}
A localizer in the sense of Definition~\ref{DEFFUNLOC} is saturated and thus satisfies 2-out-of-6 and is closed under retracts. 
\end{SATZ}

We also show in the end of the section that there is always an associated left derivator. 
\begin{proof}
By Proposition~\ref{PROPPROPERTIESLOCALIZER}, 5.\@ it suffices to see that the image of $\mathcal{W}$ in $\tau_1 \Dia(\mathcal{S})$ is saturated. 
This follows from Lemma~\ref{LEMMASAT} below, in view of Proposition~\ref{PROPCALCFRACTIONS} and Lemma~\ref{LEMMA2OUTOF6}.
\end{proof}

\begin{LEMMA} \label{LEMMAPUSHOUTW}
Let $\mathcal{W}$ be a localizer in $\Dia(\mathcal{S})$ and consider a diagram $X: \lefthalfcap \rightarrow \Dia(\mathcal{S})$
\[ \xymatrix{ Y \ar[r]^g \ar[d]_f & W  \\
 Z   }\]
If $f$ is in $\mathcal{W}$ then also the natural inclusion $\iota_3: W \rightarrow \int X$ is in $\mathcal{W}$.
\end{LEMMA}
\begin{proof}
The morphism
\[ \int \left( \vcenter{ \xymatrix{ Y \ar@{=}[d] \ar[r]^g & W \\
 Y &   } } \right) \rightarrow \int \left( \vcenter{
 \xymatrix{ Y \ar[r]^g \ar[d]_f & W \\
 Z & } } \right) \]
 is a weak equivalence by Corollary~\ref{KORGCONSTRUCTION}. 
 Furthermore there are adjunctions in the 2-category $\Dia(\mathcal{S})$
 \[ \int \left( \vcenter{ \xymatrix{ Y \ar@{=}[d] \ar[r]^g & W \\
 Y &   } } \right) \leftrightarrow \int \left( \xymatrix{ Y \ar[r]^g & W }  \right)   \leftrightarrow W.   \]
 of the form considered in Proposition~\ref{PROPPROPERTIESLOCALIZER}, 2. The statement follows. 
\end{proof}

\begin{PROP} \label{PROPCALCFRACTIONS}
Let $\mathcal{W}$ be a localizer in $\Dia(\mathcal{S})$. Then the image of $\mathcal{W}$ in $\tau_1\Dia(\mathcal{S})$ has a left calculus of fractions.
\end{PROP}
\begin{proof}
We have to show the following two properties:
\begin{enumerate}
\item Let $f: Y \rightarrow Z$ and $g: Y \rightarrow W$ be morphisms in $\Dia(\mathcal{S})$ with $f \in \mathcal{W}$. Then there exists a diagram
\[ \xymatrix{ Y \ar[r]^g \ar[d]_f & W \ar[d]^{F} \\
 Z  \ar[r]_{G} & X   }\]
which commutes in $\tau_1(\Dia(\mathcal{S}))$ and in which $F \in \mathcal{W}$;
\item Let $f, g: Z \rightarrow X$ be two morphisms in $\Dia(\mathcal{S})$ and let $h: Y \rightarrow Z$ be a morphism in $\mathcal{W}$ such that $fh = gh$ in $\tau_1\Dia(\mathcal{S})$. Then there exists $\rho: X \rightarrow X' \in \mathcal{W}$ such that $\rho f  =\rho g$ in    $\tau_1(\Dia(\mathcal{S}))$.
\end{enumerate}

1.\@ First observe that by Proposition~\ref{PROPPROPERTIESLOCALIZER}, 5.\@ the image of $\mathcal{W}$ in $\tau_1\Dia(\mathcal{S})$ satisfies again 2-out-of-3. 
Note that $f$ and $g$ assemble to a diagram $X: \lefthalfcap \rightarrow \Dia(\mathcal{S})$. The diagram
\[ \xymatrix{ Y \ar[r]^g \ar[d]_f & W \ar[d]^{\iota_3} \\
 Z  \ar[r]_{\iota_1} & \int X   }\] 
 is  commutative in $\tau_1\Dia(\mathcal{S})$ because of the chain of 2-morphisms
 \[ \iota_1 f \Leftarrow \iota_2 \Rightarrow \iota_3 g \]
 where $\iota_2: Y \rightarrow \int X$ is the natural inclusion.  Moreover, by Lemma~\ref{LEMMAPUSHOUTW}, $\iota_3$ is again in $\mathcal{W}$. 
 
2.\@
 Since $fh = gh$ in $\tau_1\Dia(\mathcal{S})$ there is a chain of natural transformations
 \begin{equation}\label{eqchainadj} fh =: \alpha_0 \Rightarrow \alpha_1 \Leftarrow \alpha_2 \Rightarrow \cdots \Leftrightarrow \alpha_n := gh  \end{equation}
which can be seen as a morphism
\[ \alpha: \int  F \rightarrow X \]
where  $F: \Xi_n \rightarrow \Dia(\mathcal{S})$ is the diagram
 \[  \xymatrix{ Z & \ar[l]_h Y \ar[r]^{\id}  & Y  & \ar[l]_{\id} Y \ar[r]^{\id} & & \cdots & \ar[r]^h & Z  }  \]
 and $\Xi_n$ denotes its underlying shape (a chain of $\Delta_1$'s). 
 Consider the composition
 \[ p: \int_{\Xi_n} F \rightarrow \int_{\Xi_n} Z \rightarrow Z  \]
 where in the middle, by abuse of notation, $Z$ denotes the constant functor with value $Z$.  The first morphism is a weak equivalence by Corollary~\ref{KORGCONSTRUCTION}. The second morphism is also a weak equivalence because $\Xi_n$ is contractible\footnote{There is an obvious sequence of adjunctions $\Xi_n =: \Xi_n' \leftrightarrow \Xi_{n-1}' \leftrightarrow \cdots \leftrightarrow \Xi_{-1}'=\{\cdot\}$ where $\Xi_{i}'$ is obtained from $\Xi_{i+1}'$ by deleting an extremal object. }.
We have a  diagram (commutative in $\tau_1 \Dia(\mathcal{S})$):
\[ \xymatrix{ \int F \ar[r]^{\alpha} \ar[d]_p & X \ar[d]^{\rho} \\
 Z \ar[r]_{\iota_3} & \int G   }\] 
 where $G: \lefthalfcap \rightarrow \Dia(\mathcal{S})$ is the diagram formed by $\alpha$ and $p$.
 Since $p$ is a weak equivalence, the morphism $\rho$ is also a weak equivalence by Lemma~\ref{LEMMAPUSHOUTW}. 
 Moreover, there are three inclusions $\iota_{1,2}: Z \rightarrow \int F \rightarrow \int G$ and $\iota_{3}: Z \rightarrow \int G$ and 2-morphisms
 \[  \rho \circ f \Leftarrow \iota_1 \Rightarrow \iota_3 \Leftarrow \iota_2 \Rightarrow  \rho \circ g. \]
 In particular
 $\rho f = \rho g$
 in $\tau_1 \Dia(\mathcal{S})$.
\end{proof}

\begin{LEMMA}\label{LEMMA2OUTOF6}
A localizer $\mathcal{W}$ satisfies 2-out-of-6.  
\end{LEMMA}
\begin{proof}The following proof is inspired by \cite{BM11}. 
It suffices to see this for the image of $\mathcal{W}$ in $\tau_1(\Dia(\mathcal{S}))$.
Let
\[ \xymatrix{A \ar[r]^f & B \ar[r]^g & C  \ar[r]^h & D  } \]
be morphisms such that $gf$ and $hg$ are in $\mathcal{W}$. Consider the diagram (commutative in $\tau_1\Dia(\mathcal{S})$)
\[ \xymatrix{A \ar[r]^f \ar[d]_{gf} & B \ar[d]^{\mu}  \\ 
C \ar[r]_-\kappa & \int F \ar[r]_-p & C
 }\]
 where $F: \lefthalfcap \rightarrow \Dia$ is the  subdiagram formed by $f$ and $gf$. By Lemma~\ref{LEMMAPUSHOUTW} the morphism $\mu$ is a weak equivalence. By assumption $h p \mu = hg$ is a weak equivalence. Thus
 $hp= h p \kappa p$ is a weak equivalence. Thus $\kappa p$ is a weak equivalence by 2-out-of-3, and since $p \kappa = \id$ (on the nose, not only in $\tau_1 \Dia(\mathcal{S})$), and $\mathcal{W}$ is weakly saturated, the morphisms $p$ and $\kappa$ are in $\mathcal{W}$.
 Thus $f$, and hence also $g$ and $h$ are weak equivalences. 
\end{proof}

We have the following well-known lemma, whose proof we include for the sake of completeness. 
\begin{LEMMA}\label{LEMMASAT}
Let $(\mathcal{C}, \mathcal{W})$ be a category with weak equivalences. If $\mathcal{W}$ satisfies 2-out-of-6 and has a calculus of left fractions then it is saturated.
\end{LEMMA}
\begin{proof}
Using the calculus of left fractions, morphisms in $\mathcal{C}[\mathcal{W}^{-1}]$ are cospans
\[ \xymatrix{X \ar[r]^{f} & Z & \ar[l]_w Y} \]
in which $w \in \mathcal{W}$, modulo the equivalence relation generated by the relation that the two extremal spans in a commutative diagram of the following shape are equivalent:
\[ \xymatrix{ & Z \ar[dd] & \\
X \ar[ru]^{f} \ar[rd]_{f'}  &  &  \ar[lu]_{w}  \ar[ld]^{w'} Y \\
& Z' } \]
From this and 2-out-of-3 it follows that any morphism which is equivalent to an identity is in $\mathcal{W}$. 
If $f: X \rightarrow Y$ is a morphism in $\mathcal{C}$ with inverse 
\[ \xymatrix{Y \ar[r]^{g} & Z & \ar[l]_w X} \]
in $\mathcal{C}[\mathcal{W}^{-1}]$ then  the following composition is equivalent to the identity:
\[ \xymatrix{  & & Z & & \\
& Y \ar[ru]^{g}  & & Z  \ar@{=}[lu]_{}  & \\
X \ar[ru]^{f}   &  &  \ar@{=}[lu]_{}  Y \ar[ru]^{g}   &  &  \ar[lu]_{w}  X } \]
Therefore $gf \in \mathcal{W}$. Also the following composition is equivalent to the identity:
\[ \xymatrix{  & & Z' & & \\
& Z \ar[ru]^{f'}  & & Y  \ar[lu]_{w'}  & \\
Y \ar[ru]^{g}   &  &  \ar[lu]_{w}  X \ar[ru]^{f}   &  &  \ar@{=}[lu]  Y } \]
and hence $f'g \in \mathcal{W}$.  By 2-out-of-6 the morphism $f$ is thus a weak equivalence. 
\end{proof}

\begin{PROP}\label{PROPDERLOCALIZER}
Let $\mathcal{W}$ be a localizer on $\Dia(\mathcal{S})$. Then the association
\[ I \mapsto \DD(I):=\Dia(\mathcal{S})^I[\mathcal{W}_I^{-1}]  \]
is a left derivator with domain $\Dia$ (in which the categories $\DD(I)$ are not necessarily locally small). Furthermore, for each diagram $I$ and functor $F: I \rightarrow \Dia(\mathcal{S})$ we have the following formula for its homotopy colimit:
\[ p_{I,!} F = \int F. \]
\end{PROP}
\begin{proof}
This is an application of \cite[Theorem A.7]{Hor21}. By Theorem~\ref{SATZSAT} the class $\mathcal{W}$ is saturated. 
We have to show that
\[ I, F \mapsto \hocolim_I F := \int F \] 
 is a transitive functorial calculus of homotopy colimits in the sense of \cite[Definitions A.1 and A.6]{Hor21}:
The functoriality morphisms of \cite[Definition A.1]{Hor21} for $\alpha: I \rightarrow J$ are the opfibrations
\[ \alpha': \int_I \alpha^* F \rightarrow \int_J F. \]
The morphism $\mathrm{can}$ of  \cite[Definition A.1]{Hor21} is the identity, and the morphisms $\Xi_F$ of  \cite[Definition A.6]{Hor21} are given by the transitivity of the Grothendieck construction. 
We have to verify the axioms of a transitive functorial calculus of homotopy colimits:

(HC1) $\int$ maps object-wise weak equivalences to weak equivalences by Corollary~\ref{KORGCONSTRUCTION}. 

(HC2) $1'$ has a formal left adjoint and is thus a weak equivalence by Proposition~\ref{PROPPROPERTIESLOCALIZER}, 5.\@ and 2-out-of-6. 

(HC3) and (HC4) are elementary properties of the Grothendieck construction.

(HC5) is vacuous because $\mathrm{can}$ is the identity. 
\end{proof}

\newpage
\bibliographystyle{abbrvnat}
\bibliography{6fu}

\begin{thebibliography}{18}
\providecommand{\natexlab}[1]{#1}
\providecommand{\url}[1]{\texttt{#1}}
\expandafter\ifx\csname urlstyle\endcsname\relax
  \providecommand{\doi}[1]{doi: #1}\else
  \providecommand{\doi}{doi: \begingroup \urlstyle{rm}\Url}\fi

\bibitem[Blumberg and Mandell(2011)]{BM11}
A.~J. Blumberg and M.~A. Mandell.
\newblock Algebraic {$K$}-theory and abstract homotopy theory.
\newblock \emph{Adv. Math.}, 226\penalty0 (4):\penalty0 3760--3812, 2011.

\bibitem[Cisinski(2003)]{Cis03}
D.-C. Cisinski.
\newblock Images directes cohomologiques dans les cat\'egories de mod\`eles.
\newblock \emph{Ann. Math. Blaise Pascal}, 10\penalty0 (2):\penalty0 195--244,
  2003.

\bibitem[Cisinski(2004)]{Cis04}
D.-C. Cisinski.
\newblock Le localisateur fondamental minimal.
\newblock \emph{Cah. Topol. G\'eom. Diff\'er. Cat\'eg.}, 45\penalty0
  (2):\penalty0 109--140, 2004.

\bibitem[Cisinski(2006)]{Cis06}
D.-C. Cisinski.
\newblock Les pr\'{e}faisceaux comme mod\`eles des types d'homotopie.
\newblock \emph{Ast\'{e}risque}, \penalty0 (308):\penalty0 xxiv+390, 2006.

\bibitem[Cisinski(2008)]{Cis08}
D.-C. Cisinski.
\newblock Propri\'et\'es universelles et extensions de {K}an d\'eriv\'ees.
\newblock \emph{Theory Appl. Categ.}, 20:\penalty0 No. 17, 605--649, 2008.

\bibitem[Gabriel and Zisman(1967)]{GZ67}
P.~Gabriel and M.~Zisman.
\newblock \emph{Calculus of fractions and homotopy theory}.
\newblock Ergebnisse der Mathematik und ihrer Grenzgebiete, Band 35.
  Springer-Verlag New York, Inc., New York, 1967.

\bibitem[Groth(2013)]{Gro13}
M.~Groth.
\newblock Derivators, pointed derivators and stable derivators.
\newblock \emph{Algebr. Geom. Topol.}, 13\penalty0 (1):\penalty0 313--374,
  2013.

\bibitem[Hirschhorn(2003)]{Hir03}
P.~S. Hirschhorn.
\newblock \emph{Model categories and their localizations}, volume~99 of
  \emph{Mathematical Surveys and Monographs}.
\newblock American Mathematical Society, Providence, RI, 2003.

\bibitem[H\"ormann(2017{\natexlab{a}})]{Hor15}
F.~H\"ormann.
\newblock Fibered multiderivators and (co)homological descent.
\newblock \emph{Theory Appl. Categ.}, 32\penalty0 (38):\penalty0 1258--1362,
  2017{\natexlab{a}}.

\bibitem[H\"ormann(2017{\natexlab{b}})]{Hor16}
F.~H\"ormann.
\newblock Derivator {S}ix-{F}unctor-{F}ormalisms --- {D}efinition and
  {C}onstruction {I}.
\newblock arXiv: \href{http://arxiv.org/abs/1701.02152}{1701.02152},
  2017{\natexlab{b}}.

\bibitem[H\"ormann(2018)]{Hor15b}
F.~H\"ormann.
\newblock Six-{F}unctor-{F}ormalisms and {F}ibered {M}ultiderivators.
\newblock \emph{Sel. Math.}, 24\penalty0 (4):\penalty0 2841--2925, 2018.

\bibitem[H\"ormann(2019)]{Hor17}
F.~H\"ormann.
\newblock Derivator {S}ix-{F}unctor-{F}ormalisms --- {C}onstruction {II}.
\newblock arXiv: \href{http://arxiv.org/abs/1902.03625}{1902.03625}, 2019.

\bibitem[H\"{o}rmann(2020)]{Hor17b}
F.~H\"{o}rmann.
\newblock Enlargement of (fibered) derivators.
\newblock \emph{J. Pure Appl. Algebra}, 224\penalty0 (3):\penalty0 1023--1063,
  2020.

\bibitem[H\"ormann(2021{\natexlab{a}})]{Hor21}
F.~H\"ormann.
\newblock Model category structures on simplicial objects.
\newblock arXiv: \href{https://arxiv.org/abs/2103.01156}{2103.01156},
  2021{\natexlab{a}}.

\bibitem[H\"ormann(2021{\natexlab{b}})]{Hor21b}
F.~H\"ormann.
\newblock Exactness of {B}ousfield localizations of simplicial pre-sheaves and
  local lifting.
\newblock arXiv: \href{https://arxiv.org/abs/2104.09975}{2104.09975},
  2021{\natexlab{b}}.

\bibitem[H\"ormann(2021{\natexlab{c}})]{Hor21d}
F.~H\"ormann.
\newblock Extension of derivator six-functor-formalisms to higher stacks.
\newblock work in progress, 2021{\natexlab{c}}.

\bibitem[Krause(2010)]{Kra09}
H.~Krause.
\newblock Localization theory for triangulated categories.
\newblock In \emph{Triangulated categories}, volume 375 of \emph{London Math.
  Soc. Lecture Note Ser.}, pages 161--235. Cambridge Univ. Press, Cambridge,
  2010.

\bibitem[Lurie(2009)]{Lur09}
J.~Lurie.
\newblock \emph{Higher topos theory}, volume 170 of \emph{Annals of Mathematics
  Studies}.
\newblock Princeton University Press, Princeton, NJ, 2009.

\end{thebibliography}

\end{document}